\documentclass{article}

\usepackage[english]{babel}

\usepackage[letterpaper,top=2cm,bottom=2cm,left=3cm,right=3cm,marginparwidth=1.75cm]{geometry}

\usepackage{amsmath,amsthm}
\usepackage{graphicx}
\usepackage{amssymb}
\usepackage{subfigure}
\usepackage{tikz}
\usepackage{ulem}
\usepackage{hyperref}
\usepackage{bm}
\usepackage{enumitem}
\usepackage{tabularx,multirow}
\newcommand{\dif}{\mathop{}\mathrm{d}}
\newcommand{\n}{\mathop{}\mathrm{n}}
\newcommand{\bbR}{\mathbb{R}}
\newcommand{\bbQ}{\mathbb{Q}}

\newcommand{\bl}{\boldsymbol}
\newcommand{\T}{\mathcal{T}}
\newcommand{\A}{\mathcal{A}}
\newcommand{\E}{\mathcal{E}}

\newcommand{\wt}{\widetilde}
\newcommand{\lal}{\langle}
\newcommand{\ral}{\rangle}
\newcommand{\pt}{\partial}
\newtheorem{thm}{Theorem}[section]%{\defn \} 
\newtheorem{lem}{Lemma}[section]

\newtheorem{re}{Remark}[section]

\usepackage{color}
\title{A boundary-corrected weak Galerkin mixed finite method for elliptic interface problems with curved interfaces}
\author{Yongli Hou, Yi Liu\footnote{Corresponding author},  Yanqiu Wang}
\begin{document}
\maketitle

\begin{abstract}
We propose a boundary-corrected weak Galerkin mixed finite element method for solving elliptic interface problems in 2D domains with curved interfaces. 
The method is formulated on body-fitted polygonal meshes, where interface edges are straight and may not align exactly with the curved physical interface. 
To address this discrepancy, a boundary value correction technique is employed to transfer the interface conditions from the physical interface to the approximate interface using a Taylor expansion approach. 
The Neumann interface condition is then weakly imposed in the variational formulation. This approach eliminates the need for numerical integration on curved elements, thereby reducing implementation complexity. 
We establish optimal-order convergence in the energy norm for arbitrary-order discretizations. 
Numerical results are provided to support the theoretical findings.

\vspace{0.2cm}
\end{abstract}

\section{Introduction}

% Let $\Omega\subset\bbR^2$ be a convex polygonal domain, and, without loss of generality, we assume that $\Omega$ is divided into two subdomains $\Omega_1$ and $\Omega_2$ by a $C^2$-smooth interface surface $\Gamma\subset\Omega$, see Fig. [] for an illustration. These subdomains contain different materials identified by a piecewise constant function $\kappa$ discontinuous across the interface $\Gamma$, i.e.,
% $$\kappa(\bl x)=\kappa_1(\bl x), \quad \text{if}\, \bl x\in\Omega_1\qquad \text{and}\qquad\kappa(\bl x)=\kappa_2(\bl x), \quad \text{if}\, \bl x\in\Omega_2.$$
% We consider the following elliptic interface problem on $\Omega$:
% \begin{equation}
%   \begin{aligned}\label{elliptic_interface_problem}
%   -\nabla\cdot(\kappa\nabla p)&=f \quad &\text{in} \,&\Omega_1\cup\Omega_2,\\
%   [p]&=g_D \quad &\text{on} \,&\Gamma,\\
%   [\kappa\nabla p\cdot\bl\n]&=g_N \quad &\text{on} \,&\Gamma,\\
%   \kappa\nabla p\cdot\bl\n&=0 \quad &\text{on} \,&\pt\Omega,\\
%   \end{aligned}
%   \end{equation}

The goal of this paper is to propose and analyze a discretization method for elliptic interface problems, which arise in various engineering and industrial applications. 
In practice, such problems often involve curved interfaces. When approximated on polygonal computational domains, the discrepancy between the physical interface  $\Gamma$ and its polygonal approximation $\Gamma_h$ introduces an additional geometric error.
The geometric discrepancy can lead to a loss of accuracy \cite{Blair_Bounds_for_the_change_in_the_solutions1973,Strang_Berger_The_change_in_solution_due_to_change_in_domain1973,Thomee_Polygonal_domain_approximation_in_Dirichlet's_problem1973}, particularly for higher-order discretizations. One approach to mitigating this issue is to reduce geometric error, as in the isoparametric finite element method \cite{isoparametric_1968,Optimal_isoparametric_finite_elements_1986} and isogeometric analysis \cite{Isogeometric_Analysis_Cottrell_2009,Isogeometric_analysis_Hughes_2005}. An alternative strategy is the boundary value correction method \cite{1972,SBM_interface,Part1}, which employs a Taylor expansion to transfer the boundary conditions from the true boundary to the approximate boundary, leading to a modified variational formulation.
Among these approaches, the shifted boundary method \cite{SBM_interface,Part1} is formulated on unfitted meshes. However, the shifted interface method \cite{SBM_interface} for the linear element has only been validated through numerical results and lacks rigorous analysis. A key advantage of the boundary value correction method is that it eliminates the need for numerical integration on curved elements, thereby reducing implementation complexity.
We also highlight recent discretizations based on the boundary value correction method, such as \cite{Burman_boundary_value_correction_2018,Gunzburger_2019,Cockburn_Boundary_conforming_DG_2009,Cockburn2012,hou_JSC,Yiliu_WG}. In recent years, discretizations on polygonal meshes have gained increasing attention due to their flexibility in approximation. 
Several numerical schemes have been proposed for polygonal meshes, including the virtual element method (VEM) \cite{vem2013,da_veiga_2019}, the discontinuous Galerkin (DG) method \cite{DG_2009,DG_1,WG_Mu_Wang}, and the weak Galerkin (WG) method \cite{WG_2015,WG_2013,WG_MFEM_elliptic}, among others. 
Although polygonal meshes have been extensively studied in recent years, relatively few results address problems involving curved boundaries or interfaces. 
Notably, the works in \cite{da_veiga_2020,Yi_Liu_nonconforming_VEM,da_veiga_2019,VEM_MFEM_curve,Guan_WG_curv_2023,Dan_Li_curve_2024,Dan_Li_curve_interface_2024,Lin_Mu_curve_2021} employ specialized techniques that are directly applicable to curved domains but require numerical integration on curved elements. 
In contrast, the boundary-corrected VEM \cite{High_order_VEM_2D3D,hou_JSC} follows the second strategy, avoiding curved mesh elements.

In this paper, we study the weak Galerkin mixed finite element method (WG-MFEM) for elliptic interface problems with curved interfaces. 
The WG method, originally proposed by Wang and Ye in \cite{WG_2013}, has since been extended to various applications \cite{Wang_DS_WG,Dan_Li_curve_2024,Dan_Li_curve_interface_2024,Yiliu_WG,Yiliu_WG_SBM,WG_newStokes,WG_elasticity,WG_MFEM_elliptic,WG_stokes}. 
In particular, WG-MFEM has been applied to second-order elliptic equations \cite{Yiliu_WG,WG_MFEM_elliptic} and the Darcy-Stokes problem \cite{Wang_DS_WG}. 
However, a theoretical gap remains in the analysis of WG-MFEM for problems with curved interfaces, which this paper aims to address. 
Furthermore, to avoid numerical integration on curved elements, we incorporate the boundary value correction technique into WG-MFEM on polygonal meshes. 
In the mixed finite element method, the Neumann interface condition
becomes essential. Since this condition depends on the outward normal vector along the interface, discrepancies between $\Gamma$ and its polygonal approximation $\Gamma_h$ arise a subtle challenge. 
It is therefore necessary to develop effective strategies to transfer the Neumann interface condition from $\Gamma$ to $\Gamma_h$.
In our approach, we apply a boundary value correction technique to map the interface data from $\Gamma$ to $\Gamma_h$.
Following \cite{Puppi_mixed}, the Neumann interface condition is then imposed weakly in the variational formulation using a penalty method. To enhance stability, we incorporate an additional term $(\nabla_w\cdot,\nabla_w\cdot)$ in the discrete formulation, inspired by \cite{Pei_Cao_XFEM_CMA}.

The paper is organized as follows. 
In Section \ref{sec2:notation}, we introduce some notation and setting;
In Section \ref{sec3:model_problem}, we describe the model problem and introduce the boundary value correction technique.
In Section \ref{sec4:discrete}, the discrete spaces and the discrete form are defined and analyzed. 
In Section \ref{sec5:err}, an optimal error estimate is proved. 
In Section \ref{sec6:numerical}, we display several numerical experiments to verify the theoretical results. 
We end in Section \ref{sec7:conclusion} with some conclusions.

\section{Notations and preliminaries}\label{sec2:notation}

We assumed that $\Omega$ is an open bounded domain in $\bbR^2$, with a convex polygonal boundary $\pt\Omega$ 
and a Lipschitz continuous and piecewise $C^2$ internal interface $\Gamma$. By `internal' we mean that $\Gamma$ is a closed curve that does not intersect with $\partial\Omega$. The interface $\Gamma$ divides $\Omega$ into two open subsets: $\Omega_1$ inside $\Gamma$ and $\Omega_2$ between $\Gamma$ and $\partial\Omega$, as shown in Fig. \ref{Fig:mesh_domain} (a). We have $\bar\Omega=\bar{\Omega}_1\cup \bar{\Omega}_2$ 
and $\Omega_1\cap\Omega_2=\emptyset$.
Let $\T_{h,i}$ be a body-fitted polygonal mesh in $\Omega_i,\,i=1,2$, that is, all the interface vertices of $\T_{h,i}$ lie on $\Gamma$ 
and all the vertices of $\T_{h,1}$ on $\Gamma$ coincide with those of $\T_{h,2}$,
as shown in Fig. \ref{Fig:mesh_domain} (b). Denote by $\Omega_{h,i}$ the polygonal region occupied by the polygonal mesh $\T_{h,i}$, 
and by $\Gamma_h=\bar\Omega_{h,1}\cap\bar\Omega_{h,2}$ the approximation interface. 
Denote by $\bl\n_h$ the normal vector pointing from $\Omega_{h,1}$ to $\Omega_{h,2}$ on $\Gamma_h$, as illustrated in Fig. \ref{Fig:mesh_domain} (b). 
Let $\T_h=\T_{h,1}\cup\T_{h,2}$, throughout the paper, we assume that the mesh $\T_h$ satisfies the shape regularity conditions proposed in \cite{WG_MFEM_elliptic}.
For each polygon $K\in\T_h$, denote by $h_K$ the diameter of $K$, and set $h:=\max_{K\in\T_h}h_K$. Let $\T_h^\Gamma$ be the set of all interface elements, that is, for $K\in \T_h^\Gamma$, we have $K\cap \Gamma\neq \emptyset$. Also, 
denote by $\E_h$ the set of all edges in $\T_h$,
by $\E_h^{\Gamma}$ the set of all interface edges in $\T_h$.
For each edge $e\in\E_h^\Gamma$, let $\wt e\subset \Gamma$ be the curved edge cut by two endpoints of $e$.
When $\Omega$ has a curved interface, the polygonal region $\Omega_{h,i}$ does not coincide
with $\Omega_i$.
We use $\Omega_h^e$ to denote a crescent-shaped region surrounded by $e$ and $\wt e$, as shown in Fig. \ref{Fig:Omega_e}. 
There are three possibilities: (a) $\Omega_h^e\subset \Omega_{1}\backslash\Omega_{h,1}$;
(b) $\Omega_h^e\subset \Omega_{h,1}\backslash\Omega_1$; (c) $\Omega_h^e=\emptyset$.
\begin{figure}[!htbp]
  \centering
  \subfigure[]
  {
  \includegraphics[height=4cm,width=6cm]{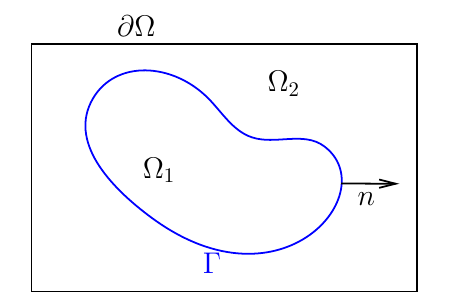}
  }\qquad
  \subfigure[]
  {
  \includegraphics[height=4cm,width=6cm]{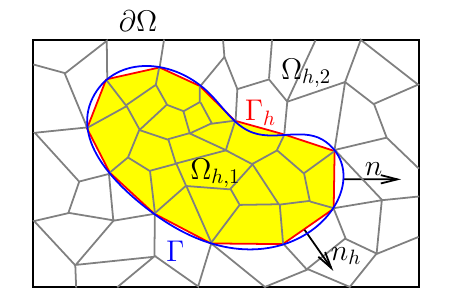}
  }
  \caption{(a). The physical domain $\Omega$ is divided into $\Omega_1$ and $\Omega_2$.
  (b). The approximation domain $\Omega_{h,1}$ and $\Omega_{h,2}$, with $\Omega_{h,2}$ shaded in yellow.}
\label{Fig:mesh_domain}
\end{figure}

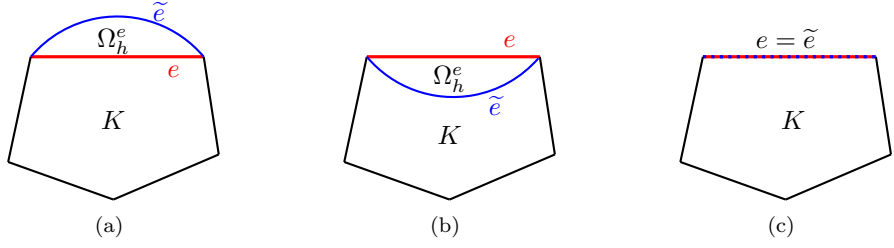
\begin{figure}[!htbp]
  \centering
  \subfigure[]
  {
  \begin{tikzpicture}
  \draw[red,very thick] (-1.3,0)--(1,0);
  \draw[thick] (-1.3,0)--(-1.6,-1.4);
  \draw[thick] (-1.6,-1.4)--(-0.2,-1.9);
  \draw[thick] (-0.2,-1.9)--(1.2,-1.3);
  \draw[thick] (1.2,-1.3)--(1,0);
  \draw[blue,thick] (1,0) arc (40:140:1.5);
  \draw (-0.2,0.5) node [below] {$\Omega_h^e$};
  \draw[red] (0.6,0) node [below] {$e$};
  \draw[] (-0.2,-0.6) node [below] {$K$};
  \draw[blue] (0.2,0.6) node [right] {$\wt e$};
  \end{tikzpicture}
  }\qquad\qquad
  \subfigure[]
  {
  \begin{tikzpicture}
  \draw[red,very thick] (-1.3,0)--(1,0);
  \draw[thick] (-1.3,0)--(-1.6,-1.4);
  \draw[thick] (-1.6,-1.4)--(-0.2,-1.9);
  \draw[thick] (-0.2,-1.9)--(1.2,-1.3);
  \draw[thick] (1.2,-1.3)--(1,0);
  \draw[blue,thick] (-1.3,0) arc (220:319:1.5);
  \draw (-0.2,0) node [below] {$\Omega_h^e$};
  \draw[red] (0.6,0) node [above] {$e$};
  \draw[] (-0.2,-0.8) node [below] {$K$};
  \draw[blue] (0.2,-0.65) node [right] {$\wt e$};
  \end{tikzpicture}
  }\qquad\qquad
  \subfigure[]
  {
  \begin{tikzpicture}
  \draw[red,very thick] (-1.3,0)--(1,0);
  \draw[blue,dotted,very thick] (-1.3,0)--(1,0);
  \draw[thick] (-1.3,0)--(-1.6,-1.4);
  \draw[thick] (-1.6,-1.4)--(-0.2,-1.9);
  \draw[thick] (-0.2,-1.9)--(1.2,-1.3);
  \draw[thick] (1.2,-1.3)--(1,0);
  \draw[] (-0.2,0.5) node [below] {$e=\wt{e}$};
  \draw[] (-0.1,-0.6) node [below] {$K$};
  \end{tikzpicture}
  }\qquad\qquad
  \caption{For $K\in\T_{h,1}$, there are three possibilities for $\Omega_h^e$. (a). The case for $\Omega_h^e\subset \Omega_1\backslash \Omega_{h,1}$.
  (b). The case for $\Omega_h^e\subset \Omega_{h,1}\backslash \Omega_1$.
  (c). The case for $e=\wt{e}$ and $\Omega_h^e=\emptyset$.}
\label{Fig:Omega_e}
\end{figure}

We use standard definitions for the Sobolev spaces \cite{FEM_Brenner}. 
Let $H^m(S)$, for $m\in\bbR$ and $S\subset\bbR^2$, be the usual Sobolev space equipped with the norm $\|\cdot\|_{m,S}$ and the seminorm $|\cdot|_{m,S}$.
When $m=0$, the space $H^0(S)$ coincides with the square integrable space $L^2(S)$.
We use bold face letters to distinguish the vector-valued function spaces
from the scalar-valued ones.
For example, $\bl H^m(S)=[H^m(S)]^2$ is the space of vector-valued functions with each component in $H^m(S)$, respectively.
Denote by $L_0^2(S)$ the subspace of $L^2(S)$ consisting of functions with mean value zero.
Define $H^m(\T_h):=\prod_{K\in\T_h}H^m(K)$ with the seminorm $|\cdot|_{m,\T_h}:=(\sum_{K\in\T_h}|\cdot|_{m,K}^2)^{1/2}$.
The above notation extends to a portion $s\subset\Gamma$ or $s\subset\Gamma_h$.
For example, $\|\cdot\|_{m,s}$ is the Sobolev norm on the curve $s$. 
We use the notation $(\cdot,\cdot)_S$ to indicate the $L^2$ inner-product on $S\subset\bbR^2$,
and $\lal\cdot,\cdot\ral_s$ to indicate the duality pair on $s\subset\Gamma$ or $s\subset\Gamma_h$. Moreover, define
\begin{equation*}
\begin{aligned}
\bl H(\mathrm{div},S)&=\{\bl v\in \bl L^2(S):\mathrm{div}\,\bl v \in L^2(S)\},\\
\bl H_0(\mathrm{div},S)&=\{\bl v\in \bl H(\mathrm{div},S):\bl v\cdot\bl\n=0\, \text{on}\,\pt S\}.\\
\end{aligned}
\end{equation*}
The norm in $\bl H(\mathrm{div},S)$ is defined by
\begin{align*}
    \|\bl v\|_{\bl H(\mathrm{div},S)}=(\|\bl v\|_{0,S}^2+\|\mathrm{div}\,\bl v\|_{0,S}^2)^{1/2}.
\end{align*}

For $S=S_1\cup S_2$, where the interior of $S_1$ and $S_2$ does not intersect, define $H^m(S_1\cup S_2)= H^m(S_1)\times H^m(S_2)$. 
When $\gamma= \bar{S}_1\cap\bar{S}_2 $ is an edge or a collection of edges, for $v\in H^m(S_1\cup S_2)$, we define the norm, jump and average on $\gamma$ by
\begin{align*}
    &\|v\|_{m,S_1\cup S_2}=\bigg(\sum_{i=1}^2\|v\|_{m,S_i}^2\bigg)^{1/2},\\
    [v] =&~ v_1-v_2 ~\text{on}~\gamma,\qquad \{v\} = \dfrac{v_1+v_2 }{2}~\text{on}~\gamma,
\end{align*}
where $v_i=v|_{\Omega_i}$,$i=1,2$.

   Throughout the paper, we use the notation $\lesssim$ to denote less than or equal to up to a constant, and the analogous notation $\gtrsim$ to denote greater than or equal to up to a constant.

   Finally, we introduce some useful inequalities.
  
   \begin{lem}\label{lem:TRACE}
    (Trace Inequality \cite{Cangiani_hp_Version,WG_Mu_Wang,WG_MFEM_elliptic}). For $K\in\T_h$, we have
    \begin{equation*}
    \begin{aligned}
    \|v\|_{0,\pt K}&\lesssim h_K^{-1/2}\| v\|_{0,K}+h_K^{1/2}|v|_{1,K},\qquad \forall v\in H^1(K).\\
    \end{aligned}
    \end{equation*}
    \end{lem}
    
    \begin{lem}\label{lem:inverse}
    (Inverse Inequality \cite{FEM_Brenner}). Given integers $0\leq m\leq l$. For $K\in\T_h$, it holds
    \begin{equation*}
    \begin{aligned}
    |q|_{l, K}\lesssim h_K^{m-l}| q|_{m,K},\qquad \forall q\in P_l(K).
    \end{aligned}
    \end{equation*}
    \end{lem}
    
    \begin{lem}\label{lem:BH}
      (Bramble-Hilbert Lemma \cite{FEM_Brenner}). Given an integer $k\geq 0$. For $K\in\T_h$ and integers $0\leq i\leq k+1$, one has
      \begin{equation*}
      \begin{aligned}
      \inf_{\xi\in P_k(K)}\bigg(\sum_{j=0}^i h_K^j|v-\xi|_{j,K}\bigg)\lesssim h_K^i|v|_{i,K},\qquad \forall v\in H^i(K).
      \end{aligned}
      \end{equation*}
    \end{lem}

%%%%%%%%%%%%%%%%%%%%%%%%%%%%%%%%%%%%%%%%%%%%%%%%%%%%%%%%%%%%%%%%%%%%%%%%%%%%%%%%%%%%%%%%%%%%%%%%%%%%%%%%%%%%%%%%%%%%%%%%%%%%
\section{Model problem and boundary-corrected method}\label{sec3:model_problem}
In this section, we introduce the model problem and briefly review the boundary value correction method \cite{1972}. 
Consider the following Darcy interface problem in $\Omega$:
\begin{equation}
  \left\{
  \begin{aligned}
  \bl u&=-\kappa\nabla p \quad &\text{in} \,&\Omega_1\cup\Omega_2 ,&\\ \label{primal_problem_interface}
  \nabla\cdot \bl u&=f \quad &\text{in} \,&\Omega_1\cup\Omega_2,\\
  \bl u\cdot\bl\n&=0 \quad &\text{on} \,&\pt\Omega,\\
  [p]&=g_D \quad &\text{on} \,&\Gamma,\\
  [\bl u\cdot\bl\n]&=g_N \quad &\text{on} \,&\Gamma,\\
  \end{aligned}
  \right.
  \end{equation}
  where $\bl u$ is the Darcy velocity, $p$ is the pressure and the unit outward normal vector $\bl\n$ on $\Gamma$ points from $\Omega_1$ into $\Omega_2$, as  illustrated in Fig. \ref{Fig:mesh_domain} (a). 
  We assume that $f\in L^2(\Omega),\,g_D\in H^{1/2}(\Gamma),\,g_N\in H^{-1/2}(\Gamma)$.
  For simplicity, assume that $\kappa_i=\kappa|_{\Omega_i},\,i=1,2$, is a positive constant in $\Omega_i$.

We use WG-MFEM \cite{Wang_DS_WG} on the polygonal mesh $\T_h$ to discretize Problem (\ref{primal_problem_interface}). 
The discretization is defined in $\Omega_{h,1}\cup\Omega_{h,2}$ rather than in $\Omega_1\cup\Omega_2$.
In the mixed formulation, the Dirichlet interface condition $[p]=g_D$ becomes natural, while the Neumann boundary condition $\bl u\cdot\bl \n|_{\partial\Omega} = 0$ and the interface condition $[\bl u\cdot\bl\n]=g_N$ becomes essential.
This poses a major difficulty in the finite element discretization, since $\Gamma$ is not equal to $\Gamma_h$. The interface condition must be transferred from $\Gamma$ to $\Gamma_h$.
In this work, we use the idea of the boundary value correction method \cite{1972} to transfer the interface condition, and design a boundary-corrected WG-MFEM for Problem \ref{primal_problem_interface}. To this end, we first introduce a Taylor expansion near the interface.

Assume that there exists a map $\rho_h:\Gamma_h\rightarrow\Gamma$ such that
    \begin{equation}
    \begin{aligned}\label{map}
    \rho_h(\bl x_h):=\bl x_h+\delta_h(\bl x_h)\bl \nu_h(\bl x_h),\qquad \forall\,\bl x_h\in\Gamma_h,
     \end{aligned}
    \end{equation}
    as shown in Fig. \ref{Fig:Mh}, where $\bl\nu_h$ denotes a unit vector pointing from $\bl x_h\in\Gamma_h$ to $\rho_h(\bl x_h)\in\Gamma$ and $\delta_h(\bl x_h)=|\rho_h(\bl x_h)-\bl x_h|$. For simplicity, 
     denote $\bl x:=\rho_h(\bl x_h)$ and $\wt{\bl\n}(\bl x_h):=\bl\n\circ \rho_h(\bl x_h)$, that is, the unit outward normal vector on $\bl x\in\Gamma$ is pulled to the corresponding $\bl x_h\in\Gamma_h$.
    Since $\Gamma$ is piecewise $C^2$ continuous, it is reasonable to further assume that $\rho$ satisfies \cite{Burman_boundary_value_correction_2018}:
    \begin{equation}
    \begin{aligned}\label{Inequality:delta}
    \delta=\sup_{\bl x_h\in\Gamma_h}\delta_h(\bl x_h)\lesssim h^2,\qquad \|\wt{\bl\n}-\bl\n_h\|_{L^\infty(\Gamma_h)}\lesssim h,
    \end{aligned}
    \end{equation}

    \begin{figure}[!htbp]
      \centering
      % \subfigure[]
      % {
      \begin{tikzpicture}
      \draw[red,very thick] (-1.71,0.8)--(1.69,0.8);
      \draw[blue, thick] (1.8,0.6) arc (25:155:2);
      \draw[gray, thick,->] (-0.3,0.8)--(-0.6,1.705);
      \draw[blue, thick,->](-0.6,1.705)--(-0.85,2.5);
      \draw (-0.1,0.7) node [below] {${\bm x_h}$};
      \draw[red] (1,0.4) node [right] {$\Gamma_h$};
      \draw[blue] (1.8,0.6) node[right] {$\Gamma$};
      \draw (0.15,1.2) node {$\delta_h(\bm x_h)$};
      \draw (-0.7,1.6) node[below] {$\bm x$};
      \draw (-0.2,2.2) node[left] {$\bl \nu_h$};
      \filldraw [black](-0.3,0.8) circle [radius=2pt];
      \filldraw [black](-0.6,1.67) circle [radius=2pt];
      \end{tikzpicture}
      % }      
      \caption{The distance vector $\delta_h(\bl x_h)$ and the unit vector $\bl{\nu_h}$ to $\Gamma_h$. 
       }
      \label{Fig:Mh}
      \end{figure}
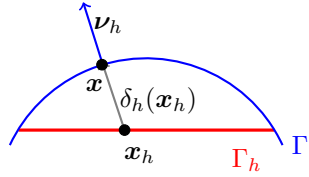

      \begin{re}
      In this paper, the map $\rho_h$ is not specified. We only require that the distance function $\delta_h(\bl x_h)$ satisfies (\ref{Inequality:delta}).
      For body-fitted meshes, this surely holds by simply setting $\bl\nu_h=\bl\n_h$. Indeed, (\ref{Inequality:delta}) can even be true on certain unfitted meshes. 
      The method and analysis in this paper extend to unfitted meshes satisfying (\ref{Inequality:delta}) with a little modification.
        \end{re}

    For a sufficiently smooth function $v$ defined between $\Gamma_h$ and $\Gamma$, using $m$-th order Taylor expansion along $\bl\nu_h$, we have
        \begin{equation*}
          \begin{aligned}
          v(\rho_h(\bl x_h))&=\sum_{j=0}^{m}\frac{\delta_h^j(\bl x_h)}{j!}\partial_{\bl\nu_h}^j v(\bl x_h)+R^m v(\bl x_h),\quad &\forall\,&\bl x_h\in\Gamma_h,\\
          %g_D(\rho_h(\bl x_h))&=[p(\bl x_h)+T_1^m p(\bl x_h)+R^s\bl p(\bl x_h)],\quad &\forall\,&\bl x_h\in\Gamma_h,\\
          \end{aligned}
          \end{equation*}
          
         where $\partial_{\bl\nu_h}^j$ is the $j$-th partial derivative in the $\bl\nu_h$ direction and the residual $R^m v(\bl x_h)$ satisfies
          $$|R^m v(\bl x_h)|=o(\delta^m).$$

          For simplicity of notation, denote 
          \begin{equation}
          \begin{aligned}\label{T1}
          T^m v(\bl x_h):=\sum_{j=0}^{m}\frac{\delta_h^j(\bl x_h)}{j!}\partial_{\bl\nu_h}^j v(\bl x_h),\qquad
          T_1^m v(\bl x_h):=\sum_{j=1}^{m}\frac{\delta_h^j(\bl x_h)}{j!}\partial_{\bl\nu_h}^j v(\bl x_h).
          \end{aligned}
          \end{equation}
Both $T^m v$ and $T_1^m v$ are functions on $\Gamma_h$.
Denote by $\wt{v}(\bl x_h):= v\circ\rho_h(\bl x_h)$ the pullback of $v$ from $\Gamma$ to $\Gamma_h$.  
As long as $v$ is smooth enough, $T^m v$ can be viewed as an approximation to $\wt{v}$, i.e.,
\begin{equation}
\begin{aligned}\label{T-R}
T^m v-\wt{v}=-R^m v,\qquad \text{on}\,\Gamma_h.
\end{aligned}
\end{equation}

Finally, we have the following two lemmas.
% {\color{red}(Have you defined $\T_h^\Gamma$ and $\E_h^\Gamma$?)} {\color{red}(Does Lem 3.1 require extra mesh condition?)}
 \begin{lem}\label{lem:Omega_he2K}
        (cf. \cite{Yiliu_WG}). Given an integer $j\geq 0$. For $K\in \T_h^\Gamma$ and $q\in P_j(K)$, one has
        \begin{equation*}
        \begin{aligned}
        \sum_{e\subset\pt K\cap\Gamma_h}\|q\|_{0,\Omega_h^e}^2\lesssim h_K\|q\|_{0,K}^2.
        \end{aligned}
        \end{equation*}
    \end{lem}
        
    \begin{lem}\label{lem:B_K1994}
         (cf. \cite{Bramble_King1994}). For each $e\in\E_h^\Gamma$ and $v\in H^1(\Omega \cup \Omega_h)$ satisfying $v|_{\Gamma}=0$, one has
        \begin{equation}
        \begin{aligned}\label{BK_e2Omeg_e}
        \|v\|_{0,e}\lesssim \delta^{1/2}|v|_{1,\Omega_h^e}.
        \end{aligned}
        \end{equation}
   \end{lem}

Lemmas \ref{lem:Omega_he2K}-\ref{lem:B_K1994} and (\ref{T1}),\,(\ref{T-R}) can be easily extended to vector-valued functions.

%%%%%%%%%%%%%%%%%%%%%%%%%%%%%%%%%%%%%%%%%%%%%%%%%%%%%%%%%%%%%%%%%%%%%%%%%%%%%%%%%%%%%%%%%%%%%%%%%%%%%%%%%%
\section{Discrete problem}\label{sec4:discrete}

In this section, we follow the general approach of WG-MFEM \cite{Wang_DS_WG} to discretize Problem (\ref{primal_problem_interface}). 
For $K\in\T_h$, denote by $K_0$ and $\pt K$ the interior and boundary of $K$, respectively. 
Given an integer $j\geq 0$, denote by $P_j(K_0)$ and $P_j(K)$ the space of polynomials with degree less than or equal to $j$. 
Corresponding vector-valued function spaces are denoted by $\bl P_j(K_0)$ and $\bl P_j(K)$.
Likewise, on each $e\in\E_h$, let $P_j(e)$ be the set of polynomials on $e$ of degree no more than $j$. Given non-negative integers $\alpha,\beta,\sigma$, assume that $\beta-1\leq \sigma\leq \beta=\alpha$. For $i=1,2$, define the weak Galerkin spaces:
\begin{equation*}
\begin{aligned}
V_{h,i}&=\{\bl v_i=\{\bl v_0,\bl v_b\}\in \bl L^2(\Omega_{h,i})\times\bl L^2(\E_{h,i}): \bl v_0|_{K_0}\in \bl P_{\alpha}(K_0), \forall\, K\in\T_{h,i},\\
&\qquad \qquad\qquad \bl v_b|_e=v_b\hspace{-0.07cm}\bl\n_e,\forall\,v_b\in P_{\beta}(e), e\in\E_{h,i}\},\\
V_{h0}&=\{\bl v=(\bl v_1,\bl v_2)\in V_{h,1}\times V_{h,2}: \bl v_b|_e=0,\forall\,e\subset \pt \Omega_h\},\\
\Psi _{h,i}&=\{q_i\in L^2(\Omega_{h,i}): q|_K\in P_{\sigma}(K),\forall\,K\in\T_{h,i}\},\\
\Psi_{h}&= \Psi_{h,1}\times \Psi_{h,2},\qquad \Psi_{h0}=\Psi_h\cap L_0^2(\Omega_h),\\
\end{aligned}
\end{equation*}
where $\bl\n_e$ is a prescribed normal direction on each edge $e\in\E_h$. 

On $K\in\T_h$, we define the weak divergence $\nabla_w\cdot\bl v\in P_{\beta}(K)$ for each $\bl v=\{\bl v_0,\bl v_b\}\in V_{h,0}$ by
\begin{align}\label{weak_div}
(\nabla_w\cdot\bl v,q)_K=-(\bl v_0,\nabla q)_K+\lal \bl v_b\cdot\bl\n_K,q\ral_{\pt K},\qquad q\in P_{\beta}(K).
\end{align}
 
 Next, we recall the extension property of Sobolev spaces:

  From \cite{book:Extension}, give a Lipschitz domain $\Omega_i$ in $\bbR^2$ and $s\in\bbR,\, s\geq 0$, for any $w\in H^{s}(\Omega_1\cup \Omega_2)$, 
  there exists an extension operator $E_i: H^s(\Omega_i)\to H^s(\bbR^2)$ such that
      \begin{align}\label{extension}
        w_i^E|_{\Omega_i}= w_i,\quad \|w_i^E\|_{s,\bbR^2}\lesssim \|w_i\|_{s,\Omega_i},\qquad \forall w_i\in H^s(\Omega_i),\,i=1,2,
      \end{align}
  where $w_i^E=E_i w_i$. %\,w_i=w|_{\Omega_i}\in H^s(\Omega_i)$.
  Moreover
      $$\|w_i^E\|_{s,\Omega_{h,i}}\leq \|w_i^E\|_{s,\bbR^2}\lesssim \|w_i\|_{s,\Omega_i}.$$
  
  Define
      \begin{equation}
        w^E=
        \left\{
        \begin{aligned}\label{extension_w}
          w_1^E\quad \text{in}\, \Omega_{h,1},\\
          w_2^E\quad \text{in}\, \Omega_{h,2}.
        \end{aligned}
        \right.
        \end{equation}
It is clear that $w\neq w^E$ on $\cup_{e\in\E_h^\Gamma} \Omega_h^e$. In $\Omega_{h,1}\backslash\Omega_1$, $w=w_2$, while $w^E=w_1^E$; 
In $\Omega_{h,2}\backslash\Omega_2$, $w=w_1$, while $w^E=w_2^E$. 
We point out that the hidden constant in $\lesssim$ of (\ref{extension}) may depend on the shape of $\Omega_i,\,i=1,2$ but not on $h$. Throughout the paper, we shall only use (\ref{extension}) on $\Omega_i$ but not on $\Omega_{h,i}$.
The extension and its notation can also be extended to vector-valued functions.
     
We then introduce a discretization for (\ref{primal_problem_interface}). Define the discrete bilinear forms $a_h: V_{h0}\times V_{h0}\rightarrow \bbR,\,b_{h1}:V_{h0}\times \Psi_{h0}\rightarrow \bbR$ and $b_{h0}:V_{h0}\times \Psi_{h0}\rightarrow \bbR$ by
\begin{equation}
  \begin{aligned}\label{bilinear}
a_h(\bl w,\bl v)&=a_0(\bl w,\bl v)+s_h(\bl w,\bl v),\\
a_0(\bl w,\bl v)&=(\kappa^{-1}\bl w_0,\bl v_0)_{\T_h}+(\nabla_w\cdot \bl w,\nabla_w\cdot \bl v)_{\T_h}+\eta\sum_{e\in\E_h^\Gamma}\lal h_K^{-1}[T^\alpha\bl w_0]\cdot\wt{\bl \n},[T^\alpha\bl v_0]\cdot\wt{\bl \n}\ral_e,\\
s_h(\bl w,\bl v)&=\rho\sum_{K\in\T_h}h_K^{-1}\lal(\bl w_0-\bl w_b)\cdot\bl\n_h,(\bl v_0-\bl v_b)\cdot\bl\n_h\ral_{\pt K},\\
b_{h1}(\bl v,q)&=-(\nabla_w\cdot\bl v,q)_{\T_h}+\sum_{e\in\E_h^\Gamma}\lal[\bl v_b\cdot\bl\n_h],\{q\}\ral_e-\sum_{e\in\E_h^\Gamma}\lal\{\bl v_b\cdot\bl\n_h\},[T_1^\sigma q]\ral_e,\\
b_{h0}(\bl v,q)&=-(\nabla_w\cdot\bl v,q)_{\T_h},
  \end{aligned}
\end{equation}
in which $\eta$ and $\rho$ are positive constants. In fact, if $\bl v$ smooth enough, it has $\bl v=\bl v_0=\bl v_b$.

The boundary-corrected weak Galerkin formulation can be written as follows: Find $(\bl u_h,p_h)\in V_{h0}\times \Psi_{h0}$ such that
\begin{equation}
  \left\{
  \begin{aligned}\label{Ph}
    a_h(\bl u_h,\bl v)+b_{h1}(\bl v,p_h)&=\ell(\bl v),\qquad&\forall\,&\bl v\in V_{h0},\\
    b_{h0}(\bl u,q_h)&=-(f^E,q_h)_{\T_h},\qquad &\forall\,&q_h\in \Psi_{h0},
  \end{aligned}
  \right.
  \end{equation}
where
$$\ell(\bl v)=(f^E,\nabla_w\cdot \bl v)_{\T_h}+\eta\sum_{e\in\E_h^\Gamma}\lal h_K^{-1}\wt g_N,[T^\alpha\bl v_0]\cdot\wt{\bl \n}\ral_e-\sum_{e\in\E_h^\Gamma}\lal \wt g_D,\{\bl v\cdot\bl\n_h\}\ral_e,$$
where $f^E$ is the extended function, $\wt g_N(\bl x_h)=g_N\circ\rho (\bl x_h)$ and $\wt g_D(\bl x_h)=g_D\circ\rho(\bl x_h)$ are pull-back of the Neumann and Dirchlet interface data from $\Gamma$ to $\Gamma_h$.

\begin{re}
$s_h(\cdot,\cdot)$ is the stabilization term of the WG discretization, and the third part in $a_0(\cdot,\cdot)$ is the penalty term obtained from interface condition weakening. 
One can view $\eta$ as a penalization parameter and view $\rho$ as a stabilization parameter,
and the parameters $\eta$ and $\rho$ in the weak Galerkin discretization can be arbitrarily chosen without affecting the approximation results. 
For brevity, in the remainder of this paper we set $\eta=\rho=1$. 
\end{re}

In the rest of this section, we shall prove the well-posdness of (\ref{Ph}), and
discuss its ``compatibility'' condition.

\subsection{Well-posedness}\label{subsec:well_posedness}

Define the energy norm on $V_{h0}$ by
$$\|\bl v\|_{0,h}=(a_h(\bl v,\bl v))^{1/2},\qquad \forall \,\bl v\in V_{h0}.$$

For $K\in\T_h$, denote by $\bl Q_0,\,\bbQ_h$ and $\pi_h$ the $L^2$ projection onto $\bl P_{\alpha}(K),\, P_\sigma(K)$ and $P_\beta(K)$, respectively.
On each $e\in\E_h$, denote by $Q_b$ the $L^2$ projection onto $P_{\beta}(e)$.
For $\bl v\in \bl H^1(\T_h)$, define $\bl Q_b\bl v=(Q_b(\bl v\cdot\bl \n_e))\bl \n_e$ for each $e\in\E_h$.
Combining these local projections together, we can define a projection $\bl Q_h=\{\bl Q_0,\bl Q_b\}$ on $V_{h0}$.
From \cite{Wang_DS_WG}, the following commutative property holds:
\begin{align}
\nabla_w\cdot(\bl Q_h \bl v)=\pi_h(\nabla\cdot\bl v),\qquad\forall \bl v\in \bl H^1(\T_h).\label{comm_diag_property}
\end{align}

Before discussing the well-posedness of the discrete problem (\ref{Ph}), we introduce some lemmas.
\begin{lem}\label{lem:T1_bounded}
(cf. \cite{hou_BDM}). Under the assumption (\ref{Inequality:delta}), for $v\in P_l(K),\,l\geq 1$, we have
\begin{align}
  \|h_K^{-1/2} T_1^m v\|_{0,e}&\leq  h|v|_{1,K},\label{T0_bounded}\\
\|h_K^{-1/2} T_1^mv\|_{0,e}&\lesssim \|v\|_{0,K}. \label{T1_bound}
\end{align}
\end{lem}

% \begin{proof}
% By the definition of $T_1^m v=\sum_{j=1}^m \frac{\delta_h^j}{j!}\partial_{\bl\nu_h}^j v$. By Lemmas \ref{lem:TRACE} and \ref{lem:inverse}, we have
% \begin{equation*}
% \begin{aligned}
%  \|h_K^{-1/2} T_1^m v\|_{0,e}&
%   =h_K^{-1/2}\left\|\sum_{j=1}^{m}\frac{\delta_h^j}{j!}\partial_{\bl\nu_h}^j v\right\|_{0,e}\\
%   &\lesssim \sum_{j=1}^{m} h^{-1/2} \delta_h^{j}\|\partial_{\bl\nu_h}^j v\|_{0,e}\\
%   &\lesssim \sum_{j=1}^{m} \left(\frac{ \delta_h}{h}\right)^{j}|v|_{1,K}\\
%   &\lesssim \left(\frac{ \delta_h}{h}\right)|v|_{1,K}\\
%   &\lesssim h|v|_{1,K},
% \end{aligned}
% \end{equation*}
% where we have used (\ref{Inequality:delta}) in the last inequality. This completes the proof of (\ref{T0_bounded}). 
% Inequality (\ref{T1_bound}) follows immediately from (\ref{T0_bounded}) and the inverse inequality in Lemma \ref{lem:inverse}.
% $\hfill\Box$\end{proof}

\begin{lem}\label{lem:vb_bound}
  Under the assumption (\ref{Inequality:delta}), for $\bl v=\{\bl v_0,\bl v_b\}\in V_{h0}$, we have
  \begin{align}
  \sum_{e\in\E_h^\Gamma}\|h_K^{1/2} \{\bl v_b\cdot\bl\n_h\}\|_{0,e}^2&\lesssim \|\bl v\|_{0,h}^2,\label{vb_bound}\\
  \sum_{e\in\E_h^\Gamma}\|h_K^{-1/2} [\bl v_b\cdot\bl\n_h]\|_{0,e}^2&\lesssim \|\bl v\|_{0,h}^2.\label{vb_jump_bound}
  \end{align}
\end{lem}

\begin{proof}
  From the trace and inverse inequalities in Lemmas \ref{lem:TRACE},\,\ref{lem:inverse}, by the definition of $s_h(\cdot,\cdot)$, one gets
  \begin{align*}
    \sum_{e\in\E_h^\Gamma}\|h_K^{1/2} \{\bl v_b\cdot\bl\n_h\}\|_{0,e}^2&\lesssim \sum_{e\in\E_h^\Gamma}h_K(\|\{(\bl v_0-\bl v_b)\cdot\bl\n_h\}\|_{0,e}^2+\|\{\bl v_0\cdot\bl\n_h\}\|_{0,e}^2)\\
    &\lesssim \sum_{K\in\T_h}h_K\|(\bl v_0-\bl v_b)\cdot\bl\n_h\|_{0,\pt K}^2+\|\bl v_0\|_{0,\T_h}^2\\
    & \lesssim \|\bl v\|_{0,h}^2.
   \end{align*}
   Thus one obtains (\ref{vb_bound}). Similarly, we have
   \begin{equation}
    \begin{aligned}\label{vb_temp1}
    \sum_{e\in\E_h^\Gamma}&\|h_K^{-1/2} [\bl v_b\cdot\bl\n_h]\|_{0,e}^2
    \lesssim \sum_{e\in\E_h^\Gamma}h_K^{-1}\bigg(\|[(\bl v_0-\bl v_b)\cdot\bl\n_h]\|_{0,e}^2+\|[\bl v_0\cdot\bl\n_h]\|_{0,e}^2\bigg)\\
    &\lesssim \sum_{K\in\T_h}h_K^{-1}\|(\bl v_0-\bl v_b)\cdot\bl\n_h\|_{0,\pt K}^2+\sum_{e\in\E_h^\Gamma}h_K^{-1}\|[\bl v_0\cdot\bl\n_h]\|_{0,e}^2.  
  \end{aligned}
  \end{equation}
   Using (\ref{Inequality:delta}),\,(\ref{T1_bound}), Lemmas \ref{lem:T1_bounded},\,\ref{lem:TRACE},\,\ref{lem:inverse}, we deduce that
   \begin{equation}
    \begin{aligned}\label{vb_temp2}
      \sum_{e\in\E_h^\Gamma}h_K^{-1}\|[\bl v_0\cdot\bl\n_h]\|_{0,e}^2&\lesssim \sum_{e\in\E_h^\Gamma}h_K^{-1}\|[\bl v_0\cdot\wt{\bl\n}]\|_{0,e}^2
      +\sum_{e\in\E_h^\Gamma}h_K^{-1}\|[\bl v_0\cdot(\wt{\bl\n}-\bl\n_h)]\|_{0,e}^2\\
    &\lesssim \sum_{e\in\E_h^\Gamma}h_K^{-1}(\|[T^\alpha\bl v_0\cdot\wt{\bl\n}]\|_{0,e}^2+\|[T_1^\alpha\bl v_0\cdot\wt{\bl\n}]\|_{0,e}^2)+\sum_{e\in\E_h^\Gamma}h\|[\bl v_0]\|_{0,e}^2\\
    &\lesssim \sum_{e\in\E_h^\Gamma}h_K^{-1}\|[T^\alpha\bl v_0\cdot\wt{\bl\n}]\|_{0,e}^2+\|\bl v_0\|_{0,\T_h}^2.
  \end{aligned}
\end{equation}
Next,  
  substituting (\ref{vb_temp2}) into (\ref{vb_temp1}), it holds
  \begin{align*}
    \sum_{e\in\E_h^\Gamma}&\|h_K^{-1/2} [\bl v_b\cdot\bl\n_h]\|_{0,e}^2\lesssim \|\bl v\|_{0,h}^2.
  \end{align*}
$\hfill\Box$\end{proof}

\begin{lem}\label{lem:bound_b}
  For $\bl u,\,\bl v\in V_{h0}$ and $q\in \Psi_{h0}$, we have
  \begin{align*}
  a_h(\bl u,\bl v)&\lesssim \|\bl u\|_{0,h}\|\bl v\|_{0,h},&\quad &a_h(\bl u,\bl v) = \|\bl v\|_{0,h}^2.\\
    b_{h1}(\bl v,q)&\lesssim \|\bl v\|_{0,h}\|q\|_{0,\Omega_h},&\quad &b_{h0}(\bl v,q)\lesssim \|\bl v\|_{0,h}\|q\|_{0,\Omega_h}.
  \end{align*}
\end{lem}

\begin{proof}
  From the definition of the energy norm, it is obvious that $a_h(\bl u,\bl v) = \|\bl v\|_{0,h}^2$. The upper bounds of $a_{h}$ and $b_{h0}$ follow immediately from the Schwarz inequality. 
  
  Using the Schwarz inequality and the Lemmas \ref{lem:TRACE},\,\ref{lem:inverse},\,\ref{lem:vb_bound} and Inequality (\ref{T1_bound}), we obtain
  \begin{align*}
   b_{h1}(\bl v,q)&\leq \|\nabla_w\cdot\bl v\|_{0,\T_h}\|q\|_{0,\T_h}+\sum_{e\in\E_h^\Gamma}\|[\bl v_b\cdot\bl\n_h]\|_{0,e}\|\{q\}\|_{0,e}+\sum_{e\in\E_h^\Gamma}\|\{\bl v_b\cdot\bl\n_h\}\|_{0,e}\|[T_1^\sigma q]\|_{0,e}\\
   &\leq\|\nabla_w\cdot\bl v\|_{0,\T_h}\|q\|_{0,\Omega_h}+\bigg(\sum_{e\in\E_h^\Gamma}h_K^{-1}\|[\bl v_b\cdot\bl\n_h]\|_{0,e}^2\bigg)^{1/2}\bigg(\sum_{e\in\E_h^\Gamma}h_K\|\{q\}\|_{0,e}^2\bigg)^{1/2}\\
   &\quad+\bigg(\sum_{e\in\E_h^\Gamma}h_K\|\{\bl v_b\cdot\bl\n_h\}\|_{0,e}^2\bigg)^{1/2}\bigg(\sum_{e\in\E_h^\Gamma}h_K^{-1}\|[T_1^\sigma q]\|_{0,e}^2\bigg)^{1/2}\\
   &\lesssim \|\bl v\|_{0,h}\|q\|_{0,\Omega_h}.
  \end{align*}
$\hfill\Box$\end{proof}

Lemma \ref{lem:bound_b} gives the boundness of $a_h(\cdot,\cdot),\,b_{h1}(\cdot,\cdot)$ and $b_{h0}(\cdot,\cdot)$. Then, the inf-sup conditions of $b_{h0}$ and $b_{h1}$ are derived in the following lemmas.
The standard way of the WG-MFEM can be used to prove the discrete inf-sup condition of $b_{h0}$. 
It seems infeasible to prove the inf-sup conditions of $b_{h1}$ for any $q\in\Psi_{h0}$ in the same way. In the following, we shall explain why. 
Firstly, this is attributed to the existence of interface terms on $b_{h1}$.
Secondly, a function $q\in \Psi_{h0}$ is mean value free on $\Omega_h$ but not necessarily on $\Omega_i,\,i=1,2$, which poses additional difficulty in the analysis.
In order to solve this issue, the main idea is to utilize the Macro-element technique (see the Theorem 1.12 of Chapter 2 in \cite{book:1986_NS}). 
% which uses the orthogonality decomposition of $\Psi_{h0}=W_h\oplus \overline{W}_h$.

\begin{lem}\label{lem:inf_sup_b0}
 (Inf-sup condition). For $q\in \Psi_{h0}$, one has
    \begin{align*}
    %\sup_{\bl v\in V_{h0}}\frac{b_{h1}(\bl v,q)}{\|\bl v\|_{0,h}}\geq C \|q\|_{1,h}.
    \sup_{\bl v\in V_{h0}}\frac{b_{h0}(\bl v,q)}{\|\bl v\|_{0,h}}\gtrsim \|q\|_{0,\Omega_h}.
    \end{align*}
\end{lem}

\begin{proof}
  For $q\in \Psi_{h0}\subset L_0^2(\Omega_h)$, it is well-known (see, e.g. \cite{FEM_Brenner}) that there exists $\bl w\in \bl H_0^1(\Omega_h)$ such that
  $\nabla\cdot\bl w=-q$ and $\|\bl w\|_{1,\Omega_h}\lesssim\|q\|_{0,\Omega_h}$, where the hidden constant in $\lesssim$ may depend on the shape of $\Omega_h$ but not on $h$. Meanwhile, note that the shape of $\Omega_h$ is the same as that of $\Omega$, and both remains unchanged. Then, let $\bl v=\bl Q_h \bl w$, by the commutative property (\ref{comm_diag_property}), we have
  \begin{equation}
  \begin{aligned}\label{b_h0_bound}
    b_{h0}(\bl Q_h{\bl w},q)&=-(\nabla_w\cdot\bl Q_h {\bl w},q)_{\T_h}=-(\pi_h(\nabla\cdot\bl w),q)_{\T_h}\\
    &=-(\nabla\cdot{\bl w},q)_{\T_h}=\|q\|_{0,\Omega_h}^2.
  \end{aligned}
  \end{equation}
  Next, we prove $\|\bl Q_h\bl w\|_{0,h}\lesssim\|q\|_{0,\Omega_h}$. From the properties of $\bl Q_h$, it is easy that
  \begin{align*}
  \|\bl Q_0\bl w\|_{0,\T_h}&\leq \|\bl w\|_{0,\Omega_h},\\
 \|\nabla_w\cdot \bl Q_h\bl w\|_{0,\T_h}&=\|\pi_h(\nabla\cdot\bl w)\|_{0,\T_h}\leq \|\nabla\cdot\bl w\|_{0,\Omega_h}\lesssim \|\nabla\bl w\|_{0,\Omega_h}.
 \end{align*}
  Using the fact that $(\bl Q_0\bl w\cdot\bl\n_h)|_e=(\bl Q_b(\bl Q_0\bl w))\cdot\bl\n_h$, the trace inequality in Lemma \ref{lem:TRACE} and the properties of $L^2$ projection, one gets
  \begin{equation*}
    \begin{aligned}
     \sum_{K\in\T_h}h^{-1}\|(\bl Q_0\bl w-\bl Q_b\bl w)\cdot\bl\n_h\|_{0,\pt K}^2&=\sum_{K\in\T_h}h^{-1}\|(\bl Q_b(\bl Q_0\bl w-\bl w))\cdot\bl\n_h\|_{0,\pt K}^2\\
     &\lesssim \sum_{K\in\T_h}h^{-1}\|\bl Q_0\bl w-\bl w\|_{0,\pt K}^2\\
     &\lesssim \|\nabla\bl w\|_{0,\Omega_h}^2.
    \end{aligned}
  \end{equation*}
Since $\bl w\in \bl H_0^1(\Omega_h)$, using the continuity of $\bl w$ in $\Omega_h$, it shows that $[\bl Q_0\bl w]|_{\Gamma_h}=0$. By the definition of
$T^\alpha$ and $T_1^\alpha$ in (\ref{T1}), one has $[T^\alpha(\bl Q_0\bl w)]|_{\Gamma_h}=[T_1^\alpha(\bl Q_0\bl w)]|_{\Gamma_h}$. 
By Lemma \ref{lem:T1_bounded} and the property of $L^2$ projection, one has
\begin{equation*}
  \begin{aligned}
   \sum_{e\in\E_h^\Gamma}h^{-1}\|[T^\alpha(\bl Q_0\bl w)]\cdot\wt{\bl\n}\|_{0,e}^2\leq &\sum_{e\in\E_h^\Gamma}h^{-1}\|[T_1^\alpha(\bl Q_0\bl w)]\|_{0,e}^2
   \lesssim\|\bl Q_0\bl w\|_{0,\T_h}^2\lesssim\|\bl w\|_{0,\Omega_h}^2.
  \end{aligned}
  \end{equation*}
  Combining the inequalities above, we obtain $\|\bl Q_h\bl w\|_{0,h}\lesssim\|\bl w\|_{1,\Omega_h}\lesssim\|q\|_{0,\Omega_h}$. Then 
  \begin{align*}
    \sup_{\bl v\in V_{h0}}\frac{b_{h0}(\bl v,q)}{\|\bl v\|_{0,h}}
    \geq \frac{b_{h0}(\bl Q_h \bl w,q)}{\|\bl Q_h \bl w\|_{0,h}}
    \gtrsim \|q\|_{0,\Omega_h}.
    \end{align*}
Thus, we complete the proof of the lemma.
$\hfill\Box$\end{proof}

\begin{figure}[!htbp]
  \centering
  {
  \includegraphics[height=4cm,width=6cm]{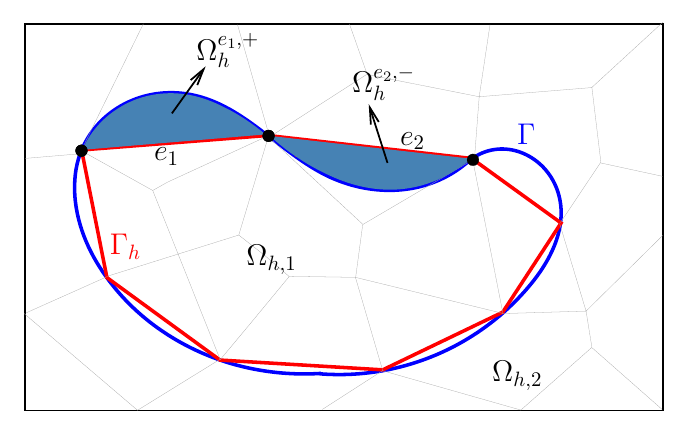}
  }
  \caption{The true boundary $\Gamma$ (blue curve), the approximated boundary $\Gamma_h$ (red lines) 
  and two typical regions $\Omega_h^{e_1,+},\Omega_h^{e_2,-}$ bounded by $\Gamma$ and $\Gamma_h$.}
  \label{Fig:yueya_interface}
  \end{figure}
  
Before we analyze the inf-sup condition of $b_{h1}$, we first introduce the following notation to distinguish between two types of crescent-shape regions $\Omega_h^e$. 
For each $e\in\E_h^\Gamma$, as shown in Fig. \ref{Fig:yueya_interface}), denote 
\[%\label{Ome_he}
\Omega_h^{e}=
\begin{cases}
\Omega_h^{e,+}, \quad\text{in} \,\Omega_1\backslash \Omega_{h,1},\\
\Omega_h^{e,-}, \quad\text{in} \,\Omega_{h,1}\backslash \Omega_1.
\end{cases}
\]
Moreover, define
\begin{equation*}
\begin{aligned}
W_{h,i}&=\Psi_{h,i}\cap L_0^2(\Omega_{h,i}),\,i=1,2,\qquad W_h=W_{h,1}\times W_{h,2},\\%\times (\Psi_{h,2}\cap L_0^2(\Omega_{h,2})),\\
\overline{W}_h&=\{q\in L_0^2(\Omega_h):q|_{\Omega_{h,i}} \text{is}\,\text{constant},\,i=1,2\}.
\end{aligned}
\end{equation*}

\begin{lem}\label{lem:inf_sup_0_b1}
When $h$ is sufficiently small, for $q\in W_h$, we have
\begin{align*}
\sup_{\bl v\in V_{h0}}\frac{b_{h1}(\bl v,q)}{\|\bl v\|_{0,h}}\gtrsim \|q\|_{0,\Omega_h}.
\end{align*}
\end{lem}

\begin{proof}
  For each $q\in W_h$ and $i=1,2$, the function $q_i=q|_{\Omega_{h,i}}\in L_0^2(\Omega_{h,i})$ can be naturally extended
  to $\Omega_i\cup \Omega_{h,i}$ by filling the gap region $\Omega_i\backslash\Omega_{h,i}$ with the same polynomial values on neighboring mesh element. 
  For brevity, we still denote their extension by $q_i$.
  Define $\widehat{q}_i=q_i-\frac{1}{|\Omega_i|}\int_{\Omega_i} q_i\dif x\in \Omega_i$.
  Since $q_i\in L_0^2(\Omega_{h,i})$, it holds
    \begin{align}\label{bar_qh2Me}
    \widehat{q}_i=q_i-\frac{1}{|\Omega_i|}\bigg(\int_{\Omega_i}q_i\dif x-\int_{\Omega_{h,i}}q_i\dif x\bigg)
    %=q_i-\frac{1}{|\Omega_i|}\int_{\Omega_i\Omega_{h,i}}q_i\dif x.
    %=q_i+(-1)^{\sigma}\frac{1}{|\Omega_i|}\sum_{e\in\E_h^\Gamma}\int_{\Omega_h^e}q_i\dif x.
    =q_i-\frac{1}{|\Omega_i|}\bigg(\int_{\Omega_i\backslash\Omega_{h,i}}q_i\dif x-\int_{\Omega_{h,i}\backslash\Omega_i}q_i\dif x\bigg).
    \end{align}
    Similar to \cite{Yiliu_WG}, one gets
    \begin{equation}
    \begin{aligned}\label{bar_qh}
    \|\widehat{q}_i\|_{0,\Omega_i}\lesssim\|q_i\|_{0,\Omega_{h,i}},\qquad (\widehat{q}_i,q_i)_{\Omega_{h,i}}=\|q_i\|_{0,\Omega_{h,i}}^2.\\
    \end{aligned}
    \end{equation}
    Again, there exists a function \cite{FEM_Brenner} $\bl w_i\in \bl H_0^1(\Omega_i)$ such that $\nabla\cdot\bl w_i=-\widehat{q}_i$ on $\Omega_i$
    and $\|\bl w_i\|_{1,\Omega_i}\lesssim \|\widehat{q}_i\|_{0,\Omega_i}\lesssim \|q_i\|_{0,\Omega_{h,i}}$, where the hidden constant in $\lesssim$ may depend on the shape of $\Omega_i$ but not on $h$.
    Extend $\bl w_i$ to $\bbR$ by setting its value to be $0$ outside $\Omega_i$. 
    The extension is still denoted by $\bl w_i$, which shall not be mistaken for $\bl w_i^E$ defined in (\ref{extension}). 
    Let $\bl v=\bl Q_h\bl w\in V_{h0}$,
    by the commutative property (\ref{comm_diag_property}), (\ref{T1}) and (\ref{bar_qh}), one has
    \begin{equation}
    \begin{aligned}\label{b_h1_bound}
    b_{h1}(\bl Q_h\bl w,q)&=-(\nabla_w\cdot\bl Q_h \bl w,q)_{\T_h}-\sum_{e\in\E_h^\Gamma}\lal \{\bl Q_b\bl w\cdot\bl\n_h\},[T_1^\sigma q]\ral_e
    +\sum_{e\in\E_h^\Gamma}\lal [\bl Q_b\bl w\cdot\bl\n_h],\{q\}\ral_e\\    
    &=-(\pi_h(\nabla\cdot\bl w),q)_{\T_h}+I_1+I_2\\
    &=\sum_{i=1}^2\bigg(-(\nabla\cdot\bl w_i,q_i)_{\Omega_i}+(\nabla\cdot\bl w_i,q_i)_{\Omega_i\backslash \Omega_{h,i}}\bigg)+I_1+I_2\\
    &=\sum_{i=1}^2\bigg((\widehat{q}_i,q_i)_{\Omega_i}-(\widehat{q}_i,q_i)_{\Omega_i\backslash \Omega_{h,i}}\bigg)+I_1+I_2\\
    &=\sum_{i=1}^2\bigg((\widehat{q}_i,q_i)_{\Omega_{h,i}}-(\widehat{q}_i,q_i)_{\Omega_{h,i}\backslash \Omega_i}\bigg)+I_1+I_2\\
    &=\|q\|_{0,\Omega_h}^2+J_1+I_1+I_2.
    \end{aligned}
    \end{equation}
   We then estimate the right-hand side of (\ref{b_h1_bound}) one by one. By (\ref{bar_qh2Me}) and Lemma \ref{lem:Omega_he2K}, one gets
   \begin{align*}
   J_1:&=-\sum_{i=1}^2(\widehat{q}_i,q_i)_{\Omega_{h,i}\backslash \Omega_i}\lesssim \sum_{i=1}^2\sum_{e\in\E_h^\Gamma}\bigg|\int_{\Omega_h^e} \widehat{q}_i\,q_i\dif x\bigg|\\
   &\lesssim \sum_{i=1}^2\sum_{e\in\E_h^\Gamma}\|q_i\|_{0,\Omega_h^e}^2+\sum_{i=1}^2\frac{1}{|\Omega_i|}\sum_{e\in\E_h^\Gamma}\bigg(\int_{\Omega_h^e} q_i\dif x\bigg)^2\\
   &\lesssim \sum_{i=1}^2h\|q_i\|_{0,\Omega_{h,i}}^2+\sum_{i=1}^2\frac{1}{|\Omega_i|}\sum_{e\in\E_h^\Gamma}|\Omega_h^e|\|q_i\|_{0,\Omega_h^e}^2\\
   &\leq O(h)\sum_{i=1}^2\|q_i\|_{0,\Omega_{h,i}}^2= O(h)\|q\|_{0,\Omega_{h}}^2.
  \end{align*}
   Using Lemmas \ref{lem:T1_bounded},\,\ref{lem:Omega_he2K},\,\ref{lem:TRACE}, Inequality (\ref{BK_e2Omeg_e}), 
   one gets
   \begin{equation}
    \begin{aligned}\label{Qbw2q}
   I_1:&=-\sum_{e\in\E_h^\Gamma}\lal \{\bl Q_b\bl w\cdot\bl\n_h\},[T_1^\sigma q]\ral_e\leq \sum_{e\in\E_h^\Gamma}\|\{\bl Q_b\bl w\cdot\bl\n_h\}\|_{0,e}\|[T_1^\sigma q]\|_{0,e}\\
  &\lesssim \bigg(\sum_{e\in\E_h^\Gamma}h\|\{\bl Q_b\bl w\cdot\bl\n_h\}\|_{0,e}^2\bigg)^{1/2}\bigg(\sum_{e\in\E_h^\Gamma}h^{-1}\|[T_1^\sigma q]\|_{0,e}^2\bigg)^{1/2}\\
  &\lesssim \bigg(\sum_{e\in\E_h^\Gamma}\sum_{i=1}^2h\|\bl w_i\|_{0,e}^2\bigg)^{1/2}\|q\|_{0,\Omega_h}\\
  &\lesssim \bigg(\sum_{e\in\E_h^\Gamma}\sum_{i=1}^2h\delta|\bl w_i|_{1,\Omega_h^e}^2\bigg)^{1/2}\|q\|_{0,\Omega_h}\\
  &\lesssim h^2\|\bl w\|_{1,\Omega}\|q\|_{0,\Omega_h}\leq O(h^2)\|q\|_{0,\Omega_h}^2.
  \end{aligned}
\end{equation}
Similarly, one obtains
   \begin{align*}
    I_2:&=\sum_{e\in\E_h^\Gamma}\lal [\bl Q_b\bl w\cdot\bl\n_h],[\{q\}\ral_e\leq \sum_{e\in\E_h^\Gamma}\|[\bl Q_b\bl w\cdot\bl\n_h]\|_{0,e}\|\{q\}\|_{0,e}\\
   &\lesssim \bigg(\sum_{e\in\E_h^\Gamma}\sum_{i=1}^2 h^{-1}\|\bl w_i\|_{0,e}^2\bigg)^{1/2}\bigg(\sum_{e\in\E_h^\Gamma}h\|\{q\}\|_{0,e}^2\bigg)^{1/2}\\
   % &\lesssim \bigg(\sum_{e\in\E_h^\Gamma}h^{-1}\delta|\bl w|_{1,\Omega_h^e}^2\bigg)^{1/2}\|q\|_{0,\Omega_h}\\
   &\lesssim h\|\bl w\|_{1,\Omega}\|q\|_{0,\Omega_h}\leq O(h)\|q\|_{0,\Omega_h}^2.
    \end{align*} 

  Combining (\ref{b_h1_bound}) and boundedness of $J_1,I_1$ and $I_2$, for $h$ sufficiently small, we have
\begin{align*}
b_{h1}(\bl Q_h\bl w,q)\geq (1-O(h))\|q\|_{0,\Omega_h}^2\gtrsim \|q\|_{0,\Omega_h}^2.
\end{align*} 

  Next, we need to prove $\|\bl Q_h\bl w\|_{0,h}\lesssim \|q\|_{0,\Omega_h}$. By the fact that $(\bl Q_0\bl w\cdot\bl\n_h)|_e=(\bl Q_b(\bl Q_0\bl w))\cdot\bl\n_h$, 
  the commutative property (\ref{comm_diag_property}), properties of $L^2$ projection and (\ref{T1}),\,(\ref{T1_bound}),\,(\ref{BK_e2Omeg_e}), yielding
  \begin{align*}
    \|\bl Q_h\bl w\|_{0,h}^2&=\|\kappa^{-1}\bl Q_0\bl w\|_{0,\T_h}^2+\|\nabla_w\cdot\bl Q_h\bl w\|_{0,\T_h}^2+\sum_{K\in\T_h}h^{-1}\|(\bl Q_0\bl w-\bl Q_b \bl w)\cdot\bl \n_h\|_{0,\pt K}^2\\
    &\quad+\sum_{e\in\E_h^\Gamma}h^{-1}\|[T^\alpha \bl Q_0\bl w]\cdot\wt{\bl\n}\|_{0,e}^2\\
    & \lesssim \|\bl w\|_{0,\T_h}^2 +\|\nabla\cdot\bl w\|_{0,\T_h}^2+\sum_{K\in\T_h}h^{-1}\|\bl Q_0\bl w-\bl w\|_{0,\pt K}^2\\
    &\quad +\sum_{e\in\E_h^\Gamma}h^{-1}(\|[\bl Q_0\bl w-\bl w]\cdot\wt{\bl\n}\|_{0,e}^2+\|[\bl w]\cdot\wt{\bl\n}\|_{0,e}^2+\|[T_1^\alpha\bl Q_0\bl w]\cdot\wt{\bl\n}\|_{0,e}^2)\\
  &\lesssim \|\bl w\|_{1,\T_h}^2+\sum_{e\in\E_h^\Gamma}h^{-1}\|[\bl w]\|_{0,e}^2+\|\bl Q_0\bl w\|_{0,\T_h}^2\\
  &\lesssim \|\bl w\|_{1,\Omega}^2+\sum_{e\in\E_h^\Gamma}h^{-1}\delta|\bl w|_{1,\Omega_h^e}^2\\
  &\lesssim \|\bl w\|_{1,\Omega}^2\lesssim \|q\|_{0,\Omega_h}^2.
\end{align*} 
This completes the proof of the lemma.
$\hfill\Box$\end{proof}

\begin{lem}\label{lem:inf_sup_C_b1}
 For $\bar{q}\in \overline{W}_h$, it holds
    \begin{align*}
    \sup_{\bl v\in V_{h0}}\frac{b_{h1}(\bl v,\bar{q})}{\|\bl v\|_{0,h}}\gtrsim\|\bar{q}\|_{0,\Omega_h}.
    %\sup_{\bl v\in V_{h0}}\frac{b_{h0}(\bl v,q)}{\|\bl v\|_{0,h}}\gtrsim_0 \|q\|_{1,h}.
    \end{align*}
\end{lem}

\begin{proof}
  For any $\bar{q}\in \overline{W}_h\subset L_0^2(\Omega_h)$, it is obvious that $[T_1^\sigma\bar{q}]|_{\Gamma_h}=0$. 
  Reiterating from \cite{FEM_Brenner}, there exists $\bl w\in \bl H_0^1(\Omega_h)$ such that $\nabla\cdot\bl w=-\bar{q}$ and $\|\bl w\|_{1,\Omega_h}\lesssim\|\bar{q}\|_{0,\Omega_h}$, where the hidden constant in $\lesssim$ is independent of $h$.
  Set $\bl v=\bl Q_h\bl w$, it implies that $[\bl Q_b\bl w]|_{\Gamma_h}=0$. Then, a similar argument of Lemma \ref{lem:inf_sup_b0} gives
  \begin{align*}
    b_{h1}(\bl Q_h\bl w,\bar{q})&=-(\nabla_w\cdot\bl Q_h {\bl w},\bar{q})_{\T_h}=\|\bar{q}\|_{0,\Omega_h}^2.
  \end{align*} 
  and
  $$\|\bl Q_h\bl w\|_{0,h}\lesssim\|\bar{q}\|_{0,\Omega_h}.$$
  Thus, we complete the proof of this lemma.
$\hfill\Box$\end{proof}

Next, using lemmas \ref{lem:inf_sup_0_b1},\,\ref{lem:inf_sup_C_b1}, we prove the inf-sup condition of $b_{h1}$ for the pair $(\bl v,q)\in V_{h0}\times\Psi_{h0}$.
\begin{lem}\label{lem:inf_sup_b1}
  (Inf-sup condition). When h is sufficiently small, for each ${q}\in \Psi_{h0}$, we obtain
    \begin{align*}
    \sup_{\bl v\in V_{h0}}\frac{b_{h1}(\bl v,{q})}{\|\bl v\|_{0,h}}\gtrsim\|{q}\|_{0,\Omega_h}.
   \end{align*}
\end{lem}

\begin{proof}
  From the definitions of spaces, we derive immediately the orthogonality of $\Psi_{h,i},\,i=1,2$ (see e.g. the Theorem 1.12 of Chapter 2 in \cite{book:1986_NS})
  $$\Psi_{h,i}=W_{h,i}\oplus\bbR.$$
  Thus each function $q\in\Psi_{h0}$ can be split 
  $$q=\widetilde{q}+\bar{q},$$
  where $\bar{q}|_{\Omega_{h,i}}=\frac{1}{|\Omega_{h,i}|}\int_{\Omega_{h,i}}q\dif x,\,i=1,2$
  and $\wt{q}\in W_h$. Observe that $\bar{q}\in \overline{W}_h$ and that the orthogonality of the decomposition implies: $$\|q\|_{0,\Omega_h}^2=\|\wt{q}\|_{0,\Omega_h}^2+\|\bar{q}\|_{0,\Omega_h}^2.$$
Owing to Lemma \ref{lem:inf_sup_0_b1}, there exists function $\bl w_i\in\bl H_0^1(\Omega_i),\,i=1,2$ such that $\|\bl w_i\|_{1,\Omega_i}\lesssim \|\wt{q_i}\|_{0,\Omega_{h,i}}$ and
$\wt{\bl v}=\bl Q_h \bl w\in V_{h0}$ satisfies
\begin{align*}
  b_{h1}(\wt{\bl v},\wt{q})\geq C_0 \|\wt{q}\|_{0,\Omega_h}^2,\qquad \|\wt{\bl v}\|_{0.h}\lesssim\|\wt{q}\|_{0,\Omega_h}.
\end{align*} 
Similarly, by Lemma \ref{lem:inf_sup_C_b1}, there exists s function $\bar{\bl v}\in V_{h0}$ such that
\begin{align*}
  b_{h1}(\bar{\bl v},\bar{q})=\|\bar{q}\|_{0,\Omega_h}^2,\qquad \|\bar{\bl v}\|_{0.h}\lesssim\|\bar{q}\|_{0,\Omega_h}.
\end{align*} 
We propose to associate with $q$ the function $\bl v=\wt{\bl v}+\eta\bar{\bl v}\in V_{h0}$ for some $\eta>0$.
We expect to adjust the parameter $\eta$ so that $(\bl v,q)\in V_{h0}\times \Psi_{h0}$ verifies the inf-sup condition of $b_{h1}$.
We obtain
\begin{align*}
  b_{h1}({\bl v},{q})&=b_{h1}(\wt{\bl v},\wt{q})+\eta b_{h1}(\bar{\bl v},\bar{q})+b_{h1}(\wt{\bl v},\bar{q})+\eta b_{h1}(\bar{\bl v},\wt{q})\\
&\geq C_0\|\wt{q}\|_{0,\Omega_h}^2+\eta\|\bar{q}\|_{0,\Omega_h}^2+b_{h1}(\wt{\bl v},\bar{q})+\eta b_{h1}(\bar{\bl v},\wt{q}).
\end{align*} 
We observe that $[T_1^\sigma\bar{q}]|_{\Gamma_h}=0$ and $\wt{\bl v}=\{\bl Q_0 \bl w,\bl Q_b \bl w\}\in V_{h0}$. By (\ref{weak_div}) and  adopting a similar argument as in (\ref{Qbw2q}), yielding 
\begin{align*}
  b_{h1}(\wt{\bl v},\bar{q})&=-(\nabla_w\cdot\wt{\bl v},\bar{q})_{\T_h}+\sum_{e\in\E_h^\Gamma}\lal [\wt{\bl v}_b\cdot\bl\n_h],\{\bar{q}\}\ral_e=\sum_{K\in\T_h}(\wt{\bl v}_0,\nabla\bar{q})_K-\sum_{e\in\E_h^\Gamma}\lal \{\wt{\bl v}_b\cdot\bl\n_h\},[\bar{q}]\ral_e\\
  &=-\sum_{e\in\E_h^\Gamma}\lal \{\bl Q_b \bl w\cdot\bl\n_h\},[\bar{q}]\ral_e\leq \sum_{e\in\E_h^\Gamma}\|\{\bl Q_b \bl w\cdot\bl\n_h\}\|_{0,e}\|[\bar{q}]\|_{0,e}\\
  &\lesssim \bigg(\sum_{e\in\E_h^\Gamma} \sum_{i=1}^2 h^{-1}\|\bl w_i\|_{0,e}^2\bigg)^{1/2}\bigg(\sum_{e\in\E_h^\Gamma} h\|\bar{q}\|_{0,e}^2\bigg)^{1/2}\\
  &\lesssim \bigg(\sum_{e\in\E_h^\Gamma} h^{-1}\delta|\bl w|_{1,\Omega_h^e}^2\bigg)^{1/2}\|\bar{q}\|_{0,\Omega_h}\\
  &\lesssim h\|\bl w\|_{1,\Omega}\|\bar{q}\|_{0,\Omega_h}\leq C_1h\|\wt{q}\|_{0,\Omega_h}\|\bar{q}\|_{0,\Omega_h}.
\end{align*} 
using the boundness of $b_{h1}$ in Lemma \ref{lem:bound_b}, one derives
\begin{align*}
  b_{h1}(\bar{\bl v},\wt{q})&\lesssim \|\bar{\bl v}\|_{0,h}\|\wt{q}\|_{0,\Omega_h}\leq C_2 \|\bar{q}\|_{0,\Omega_h}\|\wt{q}\|_{0,\Omega_h}.
\end{align*} 
Hence, collecting these results, one obtains
\begin{align*}
  b_{h1}({\bl v},{q})&\geq C_0 \|\wt{q}\|_{0,\Omega_h}^2+\eta\|\bar{q}\|_{0,\Omega_h}^2-C_1h\|\wt{q}\|_{0,\Omega_h}\|\bar{q}\|_{0,\Omega_h}-\eta C_2\|\bar{q}\|_{0,\Omega_h}\|\wt{q}\|_{0,\Omega_h}\\
  &\geq \bigg(C_0-\frac{C_1}{2}h-\eta C_2\epsilon\bigg)\|\wt{q}\|_{0,\Omega_h}^2+\eta\bigg(1-\frac{C_1}{2}h-C_2\epsilon^{-1}\bigg)\|\bar{q}\|_{0,\Omega_h}^2.
\end{align*} 
with a choice $\epsilon=2 C_2$ and $\eta=\frac{C_0}{4C_2^2}$ and when $h$ is sufficiently small, we deduce that
\begin{align*}
  b_{h1}({\bl v},{q})\gtrsim \|\wt{q}\|_{0,\Omega_h}^2+\|\bar{q}\|_{0,\Omega_h}^2=\|{q}\|_{0,\Omega_h}^2.
\end{align*} 
Finally, it holds
\begin{align*}
 \|\bl v\|_{0,h}&\leq \|\wt{\bl v}\|_{0,h}+\eta\|\bar{\bl v}\|_{0,h}
\lesssim \|\wt{q}\|_{0,\Omega_h}+\eta\|\bar{q}\|_{0,\Omega_h}\lesssim \|{q}\|_{0,\Omega_h}.
\end{align*} 
$\hfill\Box$\end{proof}

From the Theorem 3.1 in \cite{well_posedness_Nicolaides} (also see \cite{book:Boffi_Brezzi}), Lemmas \ref{lem:bound_b},\,\ref{lem:inf_sup_b0} and \ref{lem:inf_sup_b1} immediately imply that

\begin{thm}
 The discrete problem (\ref{Ph}) has a unique solution and 
 \begin{equation}
\begin{aligned}\label{A_h_bound}
\|\bl\sigma_h,\zeta_h\|_H\lesssim \sup_{(\bl v, q)\in V_{h0}\times \Psi_{h0}}\frac{\A_h((\bl\sigma_h,\zeta_h),(\bl v, q))}{\|\bl v, q\|_H},
\qquad \forall\,(\bl\sigma_h,\zeta_h)\in V_{h0}\times \Psi_{h0},
\end{aligned}
\end{equation}
where
\begin{equation}
\begin{aligned}\label{A}
% \begin{align}
 \A((\bl\sigma_h,\zeta_h),(\bl v, q))&:=a_h(\bl\sigma_h,\bl v)+b_{h1}(\bl v,\zeta_h)+b_{h0}(\bl\sigma_h,q),\\
 % \end{align} 
\|\bl\sigma_h,\zeta_h\|_H&:=(\|\bl\sigma_h\|_{0,h}^2+\|\zeta_h\|_{0,\Omega_h}^2)^{1/2}.
\end{aligned}
\end{equation}
\end{thm}

\subsection{Compatibility of (\ref{Ph}) and an equvalent discrete form}
We note that the space $\Psi_{h0}$ is not commonly used in the practical implementation. It is more convenient to construct a set of basis for
$\Psi_h$ than for $\Psi_{h0}$. Hence, we introduce a new problem: Find $(\bl u_h,p_h)\in V_{h0}\times \Psi_h$ such that
\begin{equation}
  \left\{
  \begin{aligned}\label{Ph_Qh}
    a_h(\bl u_h,\bl v)+b_{h1}(\bl v,p_h)&=\ell(\bl v),\qquad&\forall\,&\bl v\in V_{h0},\\
    b_{h0}(\bl u,q_h)+\sum_{e\in\E_h^\Gamma}\lal[\bl v_b\cdot\bl\n_h],\{\bar{q}_h\}\ral_e&=-(f^E-\overline{f^E},q_h)_{\T_h},\qquad &\forall\,&q_h\in \Psi_{h},
  \end{aligned}
  \right.
  \end{equation}
  where $\bar{q}_h=\frac{1}{|\Omega_h|}\int_{\Omega_h}q_h\dif x$ and $\overline{f^E}=\frac{1}{|\Omega_h|}\int_{\Omega_h}f^E\dif x$.

\begin{lem}\label{Ph=Ph_Qh}
The systems (\ref{Ph}) and (\ref{Ph_Qh}) are equivalent in the sense that they admit the
same solution in $V_{h0}\times \Psi_{h0}$.
\end{lem}

\begin{proof}
Evidently, the space $\Psi_h$ can be decomposed into $\Psi_{h0}\oplus \bbR$, and the decomposition is orthogonal in the $L^2(\Omega_h)$-norm. Using the definition (\ref{weak_div}), one can immediately observe that $b_{h1}(\bl v_h,1)=0$ for all $\bl v_h\in V_{h0}$. Using $q_h-\bar q_h\in\Psi_{h0}$, the second equation of (\ref{Ph}) implies that, for all $q_h\in\Psi_h$
$$b_{h0}(\bl u_h,q_h-\bar q_h)=-(f^E,q_h-\bar q_h)_{\T_h},$$
then, it is arranged to be
\begin{equation*}
\begin{aligned}
  b_{h0}(\bl u_h,q_h)+\sum_{e\in\E_h^\Gamma}\lal[\bl v_b\cdot\bl\n_h],\{\bar{q}_h\}\ral_e=-(f^E-\overline{f^E},q_h)_{\T_h},
\end{aligned}
\end{equation*}
which is just the second equation of (\ref{Ph_Qh}).
$\hfill\Box$\end{proof}

Lemma \ref{Ph=Ph_Qh} illustrates that 
as long as (\ref{Ph}) is well-posed, problem (\ref{Ph_Qh}) is solvable, i.e., the discrete linear system turns out to be ``compatible''. However, its solution is not unique since $(\bl u_h, p_h+\text{constant})$ is also a solution. Various ways can be used to deal with this case, such as, utilizing a Krylov subspace iterative solver, or imposing a constraint to the linear system to ensure that the solution stays in $V_{h0}\times\Psi_{h0}$.
 %%%%%%%%%%%%%%%%%%%%%%%%%%%%%%%%%%%%%%%%%%%%%%%%%%%%%%%%%%%%%%%%
\section{Error analysis}\label{sec5:err}
 In this section, we analyze the approximation error of the weak Galerkin discretization (\ref{Ph}). Let $\bl u,\,p$ be the
solution to the problem (\ref{primal_problem_interface}), and $\bl u_h=\{\bl u_0,\bl u_b\}$, $p_h$ be the solution to the weak Galerkin formulation (\ref{Ph}).
 Define
 $$\bl e_u=\bl Q_h \bl u^E-\bl u_h=\{\bl Q_0 \bl u^E-\bl u_0,\bl Q_b \bl u^E-\bl u_b\},\qquad e_p=\bbQ_h p^E-p_h.$$

The error equation is stated as follows:
 \begin{lem}\label{lem:err_equation}
  The exact solution $(\bl u,p)$ of system (\ref{primal_problem_interface}) satisfies
  \begin{align}\label{consistency_err}
   \A(\{\bl e_u,e_p\},\{\bl v,q\})=E_u+s_h(\bl Q_h \bl u^E,\bl v)+E_p+E_R,
  \end{align}
  where $s_h(\cdot,\cdot)$ is the stabilization term and the terms $E_u,E_p$ and $E_R$ are defined by
  \begin{align*}
    E_u:&=(\kappa^{-1} (\bl Q_0\bl u^E-\bl u^E),\bl v_0)_{\T_h}+\sum_{e\in\E_h^\Gamma}h_K^{-1}\lal[T^\alpha (\bl Q_0 \bl u^E- \bl u^E)]\cdot\wt{\bl\n},[T^\alpha\bl v_0]\cdot\wt{\bl\n}\ral_e,\\
   E_p:&=\sum_{K\in\T_h}\lal (\bl v_0-\bl v_b)\cdot\bl\n_h,\bbQ_h p^E-p^E\ral_{\pt K}+\sum_{e\in\E_h^\Gamma}\lal[\bl v_b\cdot\bl\n_h],\{\bbQ_h p^E-p^E\}\ral_e\\
   &\quad-\sum_{e\in\E_h^\Gamma}\lal\{\bl v_b\cdot\bl\n_h\},[T_1^\sigma(\bbQ_h p^E-p^E)]\ral_e,\\
   E_R:&=(\kappa^{-1}\bl u^E+\nabla p^E,\bl v_0)_{\T_h}+(\nabla\cdot \bl u^E-f^E,\nabla_w\cdot\bl v)_{\T_h}-(\nabla\cdot \bl u^E-f^E,q)_{\T_h}\\
   &\quad+\sum_{e\in\E_h^\Gamma}\bigg(h^{-1}\lal[T^\alpha \bl u^E]\cdot\wt{\bl\n}-\wt g_N,[T^\alpha\bl v_0]\cdot\wt{\bl\n}\ral_e+\lal\wt g_D-[T^\sigma p^E],\{\bl v_b\cdot\bl\n_h\}\ral_e\bigg).
  \end{align*} 
 \end{lem}

 \begin{proof}
 From the definition (\ref{A}), one obtains
 $$\A(\{\bl e_u,e_p\},\{\bl v,q\})=\A(\{\bl Q_h \bl u^E,\bbQ_h p^E\},\{\bl v,q\})-\A(\{\bl u_h,p_h\},\{\bl v,q\}).$$
 We first deduce the first term on the right side of the equation.
 Applying the commutative property (\ref{comm_diag_property}), it holds
  $$b_{h0}(\bl Q_h\bl u^E,q)=-(\nabla_w\cdot\bl Q_h\bl u^E,q)_{\T_h}=-(\pi_h(\nabla\cdot\bl u^E),q)_{\T_h}=-(\nabla\cdot\bl u^E,q)_{\T_h}.$$  
  Again using the commutative property (\ref{comm_diag_property}) and $\nabla_w\cdot \bl u^E = \nabla\cdot\bl u^E$, by subtracting and adding $\bl u^E$, one gets
  \begin{align*}
 a_0(\bl Q_h\bl u^E,\bl v)&=(\kappa^{-1} (\bl Q_0\bl u^E-\bl u^E),\bl v_0)_{\T_h}+(\kappa^{-1} \bl u^E,\bl v_0)_{\T_h}+(\nabla\cdot \bl u^E,\nabla_w\cdot\bl v)_{\T_h}\\
 &\quad+\sum_{e\in\E_h^\Gamma}h_K^{-1}\lal[T^\alpha (\bl Q_0 \bl u^E-\bl u^E)]\cdot\wt{\bl\n},[T^\alpha\bl v_0]\cdot\wt{\bl\n}\ral_e
 +\sum_{e\in\E_h^\Gamma}h_K^{-1}\lal[T^\alpha \bl u^E]\cdot\wt{\bl\n},[T^\alpha\bl v_0]\cdot\wt{\bl\n}\ral_e.
  \end{align*} 
 From (A.2) in \cite{Yiliu_WG}, one has
 \begin{align*}
  -(\nabla_w\cdot\bl v,\bbQ_hp^E)_K&=(\bl v_0,\nabla p^E)_K-\lal\bl v_0\cdot\bl\n_h,p^E\ral_{\pt K}+\lal(\bl v_0-\bl v_b)\cdot\bl\n_h,\bbQ_hp^E\ral_{\pt K}\\
&=(\bl v_0,\nabla p^E)_K+\lal(\bl v_0-\bl v_b)\cdot\bl\n_h,\bbQ_hp^E-p^E\ral_{\pt K}-\lal\bl v_b\cdot\bl\n_h,p^E\ral_{\pt K}.
\end{align*} 
 then, by the definitions of $T^\sigma$, the term $b_{h1}(\bl v,\bbQ_hp^E)$ can be written 
 \begin{align*}
  b_{h1}(\bl v,\bbQ_hp^E)&=-(\nabla_w\cdot\bl v,\bbQ_hp^E)_{\T_h}+\sum_{e\in\E_h^{\Gamma}}\lal[\bl v_b\cdot\bl\n_h],\{\bbQ_hp^E\}\ral_e
  -\sum_{e\in\E_h^\Gamma}\lal\{\bl v_b\cdot\bl\n_h\},[T_1^\sigma\bbQ_hp^E]\ral_e\\
  &=(\bl v_0,\nabla p^E)_{\T_h}+\sum_{K\in\T_h}\lal(\bl v_0-\bl v_b)\cdot\bl\n_h,\bbQ_h p^E-p^E\ral_{\pt K}-\sum_{K\in\T_h}\lal\bl v_b\cdot\bl\n_h,p^E\ral_{\pt K}\\
  &\quad+\sum_{e\in\E_h^\Gamma}\lal[\bl v_b\cdot\bl\n_h],\{\bbQ_hp^E\}\ral_e-\sum_{e\in\E_h^\Gamma}\lal\{\bl v_b\cdot\bl\n_h\},[T_1^\sigma\bbQ_hp^E]\ral_e\\
  &=(\bl v_0,\nabla p^E)_{\T_h}+\sum_{K\in\T_h}\lal(\bl v_0-\bl v_b)\cdot\bl\n_h,\bbQ_h p^E-p^E\ral_{\pt K}-\sum_{e\in\E_h^\Gamma}\lal\{\bl v_b\cdot\bl\n_h\},[p^E]\ral_e\\
  &\quad+\sum_{e\in\E_h^\Gamma}\lal[\bl v_b\cdot\bl\n_h],\{\bbQ_hp^E-p^E\}\ral_e-\sum_{e\in\E_h^\Gamma}\lal\{\bl v_b\cdot\bl\n_h\},[T_1^\sigma\bbQ_hp^E]\ral_e\\
  &=(\bl v_0,\nabla p^E)_{\T_h}+\sum_{K\in\T_h}\lal(\bl v_0-\bl v_b)\cdot\bl\n_h,\bbQ_hp^E-p^E\ral_{\pt K}\\
  &\quad+\sum_{e\in\E_h^\Gamma}\lal[\bl v_b\cdot\bl\n_h],\{\bbQ_hp^E-p^E\}\ral_e
  -\sum_{e\in\E_h^\Gamma}\lal\{\bl v_b\cdot\bl\n_h\},[T^\sigma p^E]\ral_e\\
   &\quad-\sum_{e\in\E_h^\Gamma}\lal\{\bl v_b\cdot\bl\n_h\},[T_1^\sigma(\bbQ_hp^E-p^E)]\ral_e,
\end{align*} 
where we have used the fact that
%\begin{align*}
  $\sum_{K\in\T_h}\lal\bl v_b\cdot\bl\n_h,p^E\ral_{\pt K}=\sum_{e\in\E_h^\Gamma}\lal\{\bl v_b\cdot\bl\n_h\},[p^E]\ral_e+\sum_{e\in\E_h^\Gamma}\lal[\bl v_b\cdot\bl\n_h],\{p^E\}\ral_e.$ 
  Thus, combining the above we obtain $\A(\{\bl Q_h \bl u^E,\bbQ_h p^E\},\{\bl v,q\})$,
  and then, subtracting the discrete problem (\ref{Ph}), this lemma is proved.
% \end{align*} 
$\hfill\Box$\end{proof}

 Next, we shall first derive the bounds for the right-hand side of (\ref{consistency_err}).
 \begin{lem}\label{lem:err_sh}
  Let $\bl u\in\bl H^{r+1}(\Omega_1\cup \Omega_2)\cap \bl H_0(\mathrm{div},\Omega),\, 0\leq r\leq \alpha$ 
  be the solution of the problem (\ref{primal_problem_interface}), then
  \begin{align*}
  |E_u|&\lesssim h^r\|\bl u\|_{r+1,\Omega_1\cup \Omega_2}\|\bl v\|_{0,h},\qquad\forall \bl v\in V_{h0},\\
    |s_h(\bl Q_h\bl u^E,\bl v)|&\lesssim h^r\|\bl u\|_{r+1,\Omega_1\cup \Omega_2}\|\bl v\|_{0,h},\qquad\forall \bl v\in V_{h0}.
  \end{align*} 
 \end{lem}

 \begin{proof}
 By the definition of $T^\alpha$, using the properties of $L^2$ projection and the trace inequality in Lemma \ref{lem:TRACE}, one gets
  \begin{align*}
    |E_u|&\leq \bigg(\|\kappa^{-1}(\bl Q_0\bl u^E-\bl u^E)\|_{0,\T_h}^2+\sum_{e\in\E_h^\Gamma}h^{-1}\|[T^\alpha(\bl Q_0\bl u^E-\bl u^E)]\|_{0,e}^2\bigg)^{1/2}\|\bl v\|_{0,h}\\
    &\lesssim \bigg(h^{2(r+1)}|\bl u^E|_{r+1,\T_h}^2+\sum_{e\in\E_h^\Gamma}h^{-1}\|[T^\alpha(\bl Q_0\bl u^E-\bl u^E)]\|_{0,e}^2\bigg)^{1/2}\|\bl v\|_{0,h}\\
   &\lesssim h^{r+1}|\bl u^E|_{r+1,\T_h}\|\bl v\|_{0,h}+\bigg(\sum_{e\in\E_h^\Gamma}\sum_{j=0}^\alpha\delta^{2j}h^{-1}|\bl Q_0\bl u^E-\bl u^E|_{j,e}^2\bigg)^{1/2}\|\bl v\|_{0,h}\\
   &\lesssim h^{r}|\bl u^E|_{r+1,\T_h}\|\bl v\|_{0,h}\lesssim h^r\|\bl u\|_{r+1,\Omega_1\cup \Omega_2}\|\bl v\|_{0,h}.
  \end{align*} 
  Note that $(\bl Q_0\bl u^E\cdot\bl\n_h)|_e=(\bl Q_b(\bl Q_0\bl u^E))\cdot\bl\n_h$, using the trace inequality, and the properties of $L^2$ projection, one has
  \begin{align*}
    |s_h(\bl Q_h \bl u^E,\bl v)|&\leq\sum_{K\in\T_h}h_K^{-1}\|(\bl Q_0\bl u^E-\bl Q_b\bl u^E)\cdot\bl\n_h\|_{0,\pt K}\|(\bl v_0-\bl v_b)\cdot\bl\n_h\|_{0,\pt K}\\
  &\lesssim \bigg(\sum_{K\in\T_h}h_K^{-1}\|\bl Q_0\bl u^E-\bl u^E\|_{0,\pt K}^2\bigg)^{1/2}\|\bl v\|_{0,h}\\
  &\lesssim h^{r}|\bl u^E|_{r+1, \T_h}\|\bl v\|_{0,h}\lesssim h^r\|\bl u\|_{r+1,\Omega_1\cup \Omega_2}\|\bl v\|_{0,h}.
 \end{align*} 
 $\hfill\Box$\end{proof}

 \begin{lem}\label{lem:err_Ep}
  Assume that $p\in H^{t+1}(\Omega_1\cup \Omega_2)\cap L_0^2(\Omega), \,0\leq t\leq\sigma$, we have
  \begin{align*}
    |E_p|\lesssim h^{t+1}\|p\|_{t+1,\Omega_1\cup \Omega_2}\|\bl v\|_{0,h},\qquad\forall \bl v\in V_{h0}.
  \end{align*} 
 \end{lem}

 \begin{proof}
  By Lemmas \ref{lem:vb_bound},\,\ref{lem:TRACE}, and the approximation property of $L^2$ projection, it has
  \begin{align*}
    |E_p|&\leq \sum_{K\in\T_h}\|(\bl v_0-\bl v_b)\cdot\bl\n_h\|_{0,\pt K}\|\bbQ_h p^E-p^E\|_{0,\pt K}\\
    &\quad+\sum_{e\in\E_h^\Gamma}h^{-1/2}\|[\bl v_b\cdot\bl\n_h]\|_{0,e}h^{1/2}\|\{\bbQ_h p^E-p^E\}\|_{0,e}\\
    &\quad +\sum_{e\in\E_h^\Gamma}h^{1/2}\|\{\bl v_b\cdot\bl\n_h\}\|_{0,e}h^{-1/2}\|[T_1^\sigma(\bbQ_h p^E-p^E)]\|_{0,e}\\
    &\lesssim \bigg(\sum_{K\in\T_h}h\|\bbQ_h p^E-p^E\|_{0,\pt K}^2+\sum_{e\in\E_h^\Gamma}h\|\{\bbQ_h p^E-p^E\}\|_{0,e}^2\\
&\qquad\qquad\qquad+\sum_{e\in\E_h^\Gamma}\sum_{j=1}^\sigma\delta^{2j}h^{-1}|[\bbQ_h p^E-p^E]|_{j,e}^2\bigg)^{1/2}\|\bl v\|_{0,h}\\
    &\lesssim h^{t+1}|p^E|_{t+1,\T_h}\|\bl v\|_{0,h}\lesssim h^{t+1}\|p\|_{t+1,\Omega_1\cup \Omega_2}\|\bl v\|_{0,h}.
  \end{align*} 
 $\hfill\Box$\end{proof}
 
\begin{lem}\label{lem:ER_1}
  Assume that $\bl u \in \bl H^{\max\{l+1,\alpha+2\}}(\Omega_1\cup \Omega_2),\,p\in H^{\max\{l+1,\sigma+2\}}(\Omega_1\cup \Omega_2)$ and $f\in H^l(\Omega_1\cup \Omega_2)$ satisfy Equation (\ref{primal_problem_interface}), then
  \begin{align}
  \|\kappa^{-1}\bl u^E+\nabla p^E\|_{0,\T_h}&\lesssim \delta^{l}\sum_{i=1}^2\|D^l(\kappa^{-1}\bl u^E+\nabla p^E)\|_{0,\Omega_{h,i}\backslash\Omega_i},\label{R1}\\
  \|\nabla\cdot\bl u^E-f^E\|_{0,\Omega_h}&\lesssim\delta^{l}\sum_{i=1}^2\|D^l(\nabla\cdot\bl u^E-f^E)\|_{0,\Omega_{h,i}\backslash\Omega_i},\label{R2}\\
  \sum_{e\in\E_h^\Gamma}h_K^{-1}\|[T^\alpha \bl u^E]\cdot\wt{\bl\n}-\wt g_N\|_{0,e}^2&\lesssim \delta^{2\alpha+2}h^{-1}\|D^{\alpha+2}\bl u\|_{0,\Omega_1\cup\Omega_2}^2,\label{R3}\\
  \sum_{e\in\E_h^\Gamma}h_K^{-1}\|\wt g_D-[T^\sigma p^E]\|_{0,e}^2&\lesssim \delta^{2\sigma+2}h^{-1}\|D^{\sigma+2}p\|_{0,\Omega_1\cup\Omega_2}^2.\label{R4}
\end{align}
\end{lem}

\begin{proof}
Inequalities (\ref{R1}) and (\ref{R2}) follow immediately from (3.14) in \cite{Burman_boundary_value_correction_2018}. The estimates of the residual (\ref{R3}) and (\ref{R4}) follow from (3.12) in \cite{Burman_boundary_value_correction_2018} and (\ref{extension}).
$\hfill\Box$\end{proof}

Combining (\ref{consistency_err}) and using Lemma \ref{lem:ER_1}, one gets the estimate of the estimate of $E_R$:
 \begin{lem}\label{lem:err_ER}
  Under the same assumption in Lemma \ref{lem:ER_1}, we have
  \begin{align*}
    |E_R|\lesssim &~\delta^{l}\sum_{i=1}^2\bigg(\|D^l(\kappa^{-1}\bl u^E+\nabla p^E)\|_{0,\Omega_{h,i}\backslash\Omega_i}^2
    +\|D^l(\nabla\cdot\bl u^E-f^E)\|_{0,\Omega_{h,i}\backslash \Omega_i}^2\bigg)^{1/2}\|\bl v_h,q\|_H\\
    &~+\bigg(\delta^{2(\alpha+1)}h^{-1}\|D^{\alpha+2}\bl u\|_{0,\Omega_1\cup\Omega_2}^2+\delta^{2(\sigma+1)}h^{-1}\|D^{\sigma+2}p\|_{0,\Omega_1\cup\Omega_2}^2\bigg)^{1/2}\|\bl v_h,q\|_H.
  \end{align*} 
\end{lem}

Now, we are able to write the error estimate:
 %\begin{lem}\label{lem:consis_err}
 \begin{thm}\label{thm:energy_err}
  Assume that $f\in H^{l}(\Omega_1\cup \Omega_2)$. Let $\bl u\in\bl H^{\max\{r+1,l+1,\alpha+2\}}(\Omega_1\cup \Omega_2)\cap \bl H_0(\mathrm{div},\Omega)$ 
  and $p\in H^{\max\{t+1,l+1,\sigma+2\}}(\Omega_1\cup \Omega_2)\cap L_0^2(\Omega)$ %, \,0\leq r\leq \alpha,\,0\leq t\leq\sigma$ 
  be the solution of the problem (\ref{primal_problem_interface}), then
    \begin{align*}
      \|\bl e_u\|_{0,h}+\|e_p&\|_{0,\Omega_h}
      \lesssim h^{\min\{r,t+1,\alpha,\sigma+1\}}(\|\bl u\|_{r+1,\Omega_1\cup \Omega_2}+\|p\|_{t+1,\Omega_1\cup \Omega_2})\\
      &+\delta^{l}\sum_{i=1}^2\bigg(\|D^l(\kappa^{-1}\bl u^E+\nabla p^E)\|_{0,\Omega_{h,i}\backslash\Omega_i}^2+\|D^l(\nabla\cdot\bl u^E-f^E)\|_{0,\Omega_{h,i}\backslash\Omega_i}^2\bigg)^{1/2}\\
      &+\delta^{\alpha+1}h^{-1/2}\|D^{\alpha+2}\bl u\|_{0,\Omega_1\cup\Omega_2}+\delta^{\sigma+1}h^{-1/2}\|D^{\sigma+2}p\|_{0,\Omega_1\cup\Omega_2}.
    \end{align*} 
    \end{thm}
   %\end{lem}

   \begin{proof}
 This theorem can be directly derived from the mixed finite element theory, Inequality (\ref{A_h_bound}), Lemmas \ref{lem:err_equation}-\ref{lem:err_ER}.
   $\hfill\Box$\end{proof}

   \begin{re}\label{re:an_important_trick}
  (An Important Trick). When $m=k$, for all $q_h\in P_k(\T_h)$, denote by $q_{h,i}$ the restriction in $\Omega_{h,i}$ of $q_h$, 
 and let $q_{h,i}^E$ be the natural extension of $q_{h,i}$ defined on $\bbR^2$.
  One gets $T^k q_{h,i}(\bl x_h)=q_{h,i}^E(\rho_h(\bl x_h))$ for all $\bl x_h\in\Gamma_h, i=1,2$, which leads to a significant simplification in the implementation process, since the computation of the Taylor expansion becomes unnecessary.
\end{re}

%%%%%%%%%%%%%%%%%%%%%%%%%%%%%%%%%%%%%%%%%%%%%%%%%%%%
\section{Numerical experiments}\label{sec6:numerical}

In this section, we present numerical results for the discrete schedule (\ref{Ph}) on two type of interfaces, as shown in Fig. \ref{Fig:cir_domain}. Body-fitted, unstructured triangular and quadrilateral meshes are used in the computation.
All experiments are done with polynomial degrees $\alpha=\beta=j$ and $\sigma=j-1$ for $j=1,2,3$. We use the technique in Remark \ref{re:an_important_trick}, which simplifies the implementation. Choosing different parameter pairs $(\kappa_1,\kappa_2)=(1,10),\,(10,1),\,(1,10^3),\,(10^3,1),\,(1,10^5)$,\, $(10^5,1)$. 

\begin{figure}[!htbp]
  \centering
  \subfigure[]
  {
  \includegraphics[height=4cm,width=4.2cm]{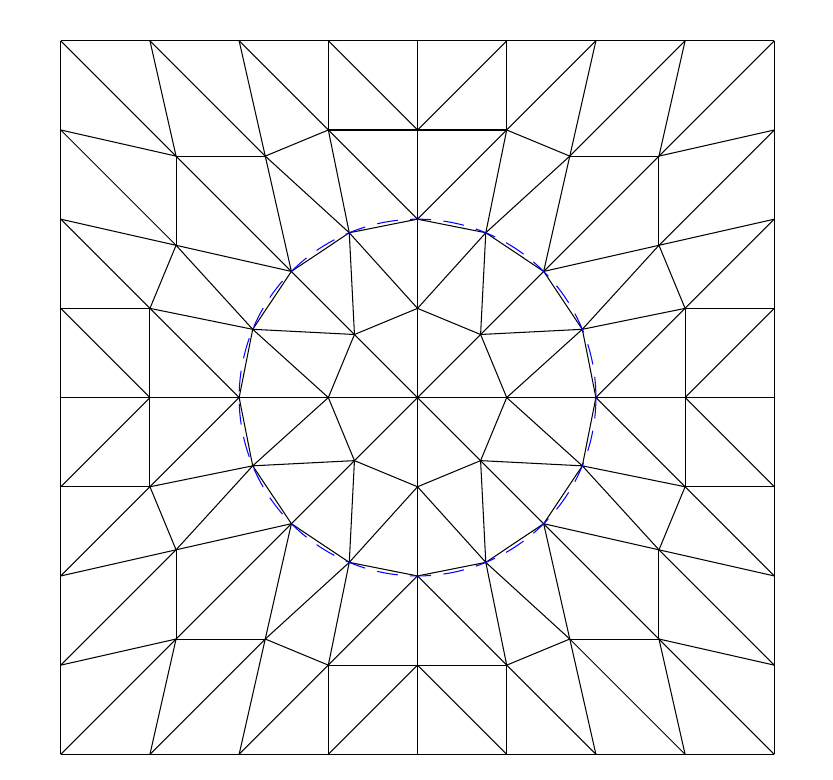}
  }\quad\qquad
  \subfigure[]
  {
  \includegraphics[height=3.85cm,width=4.0cm]{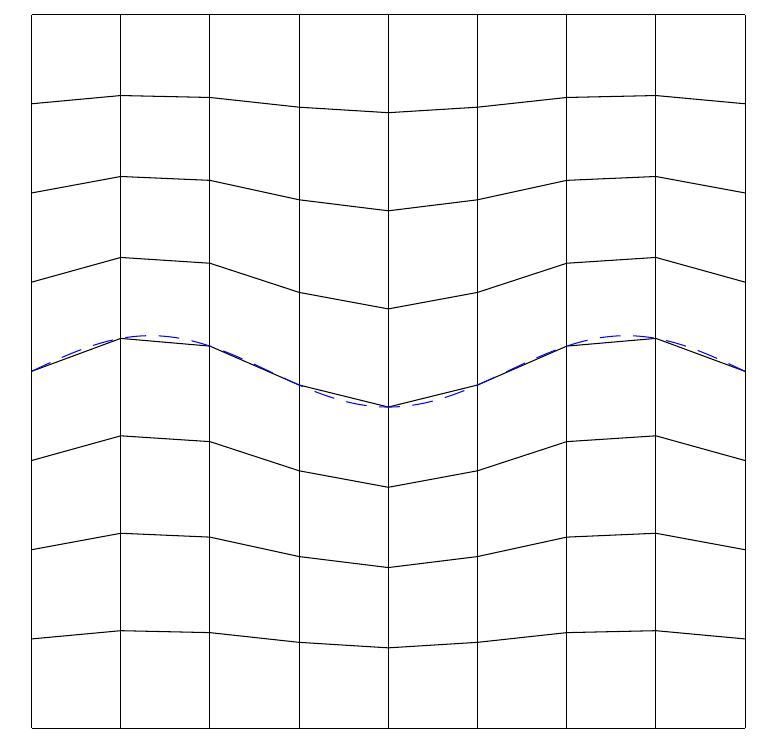}
  }
  \caption{(a). An example of triangular meshes ($h$=1/4) on domain with circular interface (blue dashed curve). (b). An example of quadrilateral mesh ($h$=1/4) on domain with curved interface $\Gamma$ (blue dashed curve).}
  \label{Fig:cir_domain}
  \end{figure}

\vspace{0.2cm}
\noindent${\bf{Example ~1. ~Circular ~interface}}$. 
\vspace{0.2cm}

We consider Problem (\ref{primal_problem_interface}) on domain $\Omega=[-1,1]\times [-1,1]$ with a circular interface $\Gamma: r^2=x^2+y^2,\,r=\frac{1}{2}$ 
and take the exact solution 
\begin{equation}
p=
  \left\{
  \begin{aligned}
  &\frac{(0.75x-x(x^2+y^2))x^2y^3}{\kappa_1}, \qquad &\text{if}\,\, x^2+y^2< r^2,\\
  &\frac{(x^2-1)^2(y^2-1)^2(x^2+y^2-0.25)^2}{\kappa_2}, \qquad &\text{if}\,\,x^2+y^2> r^2.
  \end{aligned}
  \right.
  \end{equation}

\vspace{0.2cm}
\noindent${\bf{Example ~2. ~Curved ~and ~wavy ~interface}}$. 
\vspace{0.2cm}

Consider the exact solution of Problem (\ref{primal_problem_interface}) on $\Omega = [0,1]\times [-\frac{1}{2},\frac12]$ 
with a curved interface $\Gamma :y =\dfrac{1}{20}\sin(3\pi x)$ 
\begin{equation}
p=
  \left\{
  \begin{aligned}
  &\frac{x^2(x-1)^2(y^2-\frac{1}{4})^2\sin(2\pi x)\sin(2\pi y)}{\kappa_1}, \qquad &\text{if}\,\, y< \frac{1}{20}\sin(3\pi x),\\
  &\frac{x^2(x-1)^2(y^2-\frac{1}{4})^2\cos(\pi x)\sin(\pi y)}{\kappa_2}, \qquad &\text{if}\,\,y> \frac{1}{20}\sin(3\pi x).
  \end{aligned}
  \right.
  \end{equation}

We first test Example 1 on meshes illustrated in Fig. \ref{Fig:cir_domain}(a), for various sizes of $h$, using the WG-MFEM scheme (\ref{Ph}) with boundary correction.
The errors and convergence orders in norms $\|\cdot\|_{0,h}$ and $\|\cdot\|_{0,\Omega_h}$ for the velocity and pressure fields are presented in Tabs. \ref{tab:eg1_cir_1}- \ref{tab:eg1_cir_3}. From the numerical results, 
an $O(h^k)$ optimal convergence is observed, which agrees well with the theoretical results in Theorem \ref{thm:energy_err}. 

To better examine this corrected method, we further test Example 2 on quadrilateral meshes shown in Fig. \ref{Fig:cir_domain}(b). Numerical results of the correction method are reported in Tabs. \ref{tab:eg2_sin_k1}-\ref{tab:eg2_sin_k3}. Again, an optimal $O(h^k)$ convergence is observed, as predicted in Theorem \ref{thm:energy_err}. Moreover, we test Example 2 with original the WG-MFEM discretization 
on the approximated interface without any boundary correction, where the interface conditions $g_D$ and $g_N$
are simply mapped to $\Gamma_h$ and weakly imposed.
%the interface conditions are imposed on $\Gamma_h$. 
The results are reported in Tab. \ref{tab:eg2_sin_k31}, where we
see an $O(h^{1/2})$ convergence rate in the energy-norm for velocity and an $O(h^2)$ convergence rate in the $L^2$-norm for pressure. This scheme exhibits a loss of accuracy as expected.

\begin{table}[!htbp]
\caption{\small\textit{Errors $\|\bl e_u\|_{0,h}$ and $\| e_p\|_{0,\Omega_h}$ of Example 1 with boundary correction and respect to different parameters $\kappa$ and $k=1$.}}
\label{tab:eg1_cir_1}
\begin{center}
{\small
\begin{tabular}{c|c | c| c| c|c}
\hline
\multirow{2}{*}{$(\kappa_1,\kappa_2)$}&\multirow{2}{*}{$h$}&\multicolumn{4}{c}{$k=1$}\\
\cline{3-6}
 &~ & $\|\bl e_u\|_{0,h}$  & order  & $\| e_p\|_{0,\Omega_h}$ & order \\
\cline{1-6}
\multirow{4}{*}{$(1,10)$}&
 1/8 & 3.60e-01   &   --     &   2.47e-01   &   -- \\
% \cline{2-8}
~& 1/16 & 1.90e-01 &   0.92   &   1.13e-01 &   1.12\\
% \cline{2-8}
~&1/32& 9.63e-02   &   0.98   &   5.51e-02   &   1.04\\
% \cline{2-8}
~&1/64& 4.83e-02   &   1.00   &   2.74e-02   &   1.01 \\
\cline{1-6}
\multirow{4}{*}{$(10,1)$}&
 1/8 & 3.74e-01   &   --     &   2.52e-01   &   --    \\
% \cline{2-8}
~& 1/16 & 1.98e-01 &   0.92   &  1.14e-01 &   1.14 \\
% \cline{2-8}
~&1/32& 1.00e-01  &   0.98   &   5.52e-02  &   1.05\\
% \cline{2-8}
~&1/64& 5.04e-02   &  0.99   &   2.73e-02   &  1.01 \\
\cline{1-6}
\multirow{4}{*}{$(1,10^3)$}&
 1/8 & 3.58e-01   &   --     &   3.26e-01   &   --   \\
% \cline{2-8}
~& 1/16 & 1.89e-01 &   0.92   &   1.70e-01 &   0.94\\
% \cline{2-8}
~&1/32& 9.58e-02  &   0.98   &   8.59e-02  &   0.98\\
% \cline{2-8}
~&1/64& 4.81e-02   &  1.00   &   4.31e-02   &  1.00  \\
\cline{1-6}
\multirow{4}{*}{$(10^3,1)$}&
 1/8 & 3.74e-01   &   --     &   2.52e-01   &   --  \\
% \cline{2-8}
~& 1/16 & 1.98e-01 &   0.92   &   1.14e-01 &   1.14 \\
% \cline{2-8}
~&1/32& 1.00e-01  &   0.98   &   5.52e-02  &   1.05\\
% \cline{2-8}
~&1/64& 5.04e-02   &  0.99   &   2.73e-02   &  1.01 \\
\cline{1-6}
\multirow{4}{*}{$(1,10^5)$}&
 1/8 & 3.58e-01   &   --     &   3.28e-01   &   -- \\
% \cline{2-8}
~& 1/16 & 1.89e-01 &   0.92   &   1.71e-01 &   0.94\\
% \cline{2-8}
~&1/32& 9.58e-02  &   0.98   &   8.67e-02  &   0.98\\
% \cline{2-8}
~&1/64& 4.81e-02   &  1.00   &   4.35e-02   &  1.00 \\
\cline{1-6}
\multirow{4}{*}{$(10^5,1)$}&
 1/8 & 3.74e-01   &   --     &   2.52e-01   &   -- \\
% \cline{2-8}
~& 1/16 & 1.98e-01 &   0.92   &  1.14e-01 &   1.14\\
% \cline{2-8}
~&1/32& 1.00e-01  &   0.98   &   5.52e-02  &   1.05\\
% \cline{2-8}
~&1/64& 5.04e-02   &  0.99   &   2.73e-02   &  1.01  \\
\cline{1-6}
\end{tabular}}
\end{center}
\end{table}

\begin{table}[!htbp]
\caption{\small\textit{Errors $\|\bl e_u\|_{0,h}$ and $\| e_p\|_{0,\Omega_h}$ of Example 1 with boundary correction and respect to different parameters $\kappa$ and $k=2$.}}
\label{tab:eg1_cir_2}
\begin{center}
{\small
\begin{tabular}{c|c | c| c| c|c}
\hline
\multirow{2}{*}{$(\kappa_1,\kappa_2)$}&\multirow{2}{*}{$h$}&\multicolumn{4}{c}{$k=2$}\\
\cline{3-6}
 &~ & $\|\bl e_u\|_{0,h}$  & order  & $\| e_p\|_{0,\Omega_h}$ & order \\
\cline{1-6}
\multirow{4}{*}{$(1,10)$}&
 1/8 &    5.40e-02   &   --     &   4.00e-02   &   --    \\
% \cline{2-8}
~& 1/16 &   1.39e-02   &   1.96   &   1.11e-02   &   1.85 \\
% \cline{2-8}
~&1/32&   3.50e-03   &   1.99   &   2.85e-03   &   1.96\\
% \cline{2-8}
~&1/64&   8.77e-04   &   2.00   &   7.18e-04   &   1.99 \\
\cline{1-6}
\multirow{4}{*}{$(10,1)$}&
 1/8 &   5.49e-02   &   --     &   4.00e-02   &   --     \\
% \cline{2-8}
~& 1/16 &   1.41e-02   &   1.96   &   1.11e-02   &   1.85 \\
% \cline{2-8}
~&1/32&   3.56e-03   &   1.99   &   2.85e-03   &   1.96\\
% \cline{2-8}
~&1/64&   8.91e-04   &   2.00   &   7.18e-04   &   1.99 \\
\cline{1-6}
\multirow{4}{*}{$(1,10^3)$}&
 1/8 &   5.40e-02   &   --     &  5.00e-02   &   --   \\
% \cline{2-8}
~& 1/16 &   1.39e-02   &   1.96   &    1.23e-02   &   2.03  \\
% \cline{2-8}
~&1/32&   3.50e-03   &   1.99   &   3.07e-03   &   2.00 \\
% \cline{2-8}
~&1/64&   8.77e-04   &   2.00   &   7.68e-04   &   2.00  \\
\cline{1-6}
\multirow{4}{*}{$(10^3,1)$}&
 1/8 &   5.49e-02   &   --     &  4.00e-02   &   --  \\
% \cline{2-8}
~& 1/16 &   1.41e-02   &   1.96   &   1.11e-02   &   1.85\\
% \cline{2-8}
~&1/32&   3.56e-03   &   1.99   &   2.85e-03   &   1.96\\
% \cline{2-8}
~&1/64&   8.91e-04   &   2.00   &   7.18e-04   &   1.99\\
\cline{1-6}
\multirow{4}{*}{$(1,10^5)$}&
 1/8 &   5.40e-02   &   --     &   5.04e-02   &   --   \\
% \cline{2-8}
~& 1/16 &   1.39e-02   &   1.96   &   1.23e-02   &   2.03\\
% \cline{2-8}
~&1/32&   3.50e-03   &   1.99   &   3.08e-03   &   2.00\\
% \cline{2-8}
~&1/64&   8.77e-04   &   2.00   &   7.70e-04   &   2.00 \\
\cline{1-6}
\multirow{4}{*}{$(10^5,1)$}&
 1/8 &   5.49e-02   &   --     &   4.00e-02   &   --\\
% \cline{2-8}
~& 1/16 &   1.41e-02   &   1.96   &   1.11e-02   &   1.85\\
% \cline{2-8}
~&1/32&   3.56e-03   &   1.99   &   2.85e-03   &   1.96\\
% \cline{2-8}
~&1/64&   8.91e-04   &   2.00   &   7.18e-04   &   1.99 \\
\cline{1-6}
\end{tabular}}
\end{center}
\end{table}

\begin{table}[!htbp]
\caption{\small\textit{Errors $\|\bl e_u\|_{0,h}$ and $\| e_p\|_{0,\Omega_h}$ of Example 1 with boundary correction and respect to different parameters $\kappa$ and $k=3$.}}
\label{tab:eg1_cir_3}
\begin{center}
{\small
\begin{tabular}{c|c | c| c| c|c}
\hline
\multirow{2}{*}{$(\kappa_1,\kappa_2)$}&\multirow{2}{*}{$h$}&\multicolumn{4}{c}{$k=3$}\\
\cline{3-6}
 &~ & $\|\bl e_u\|_{0,h}$  & order  & $\| e_p\|_{0,\Omega_h}$ & order \\
\cline{1-6}
\multirow{4}{*}{$(1,10)$}&
 1/8 & 4.58e-03   &   --   &   7.84e-03   &   -- \\
% \cline{2-8}
~& 1/16 & 6.25e-04   &   2.88   &   1.03e-03   &   2.93\\
% \cline{2-8}
~&1/32& 8.02e-05   &   2.96   &   1.30e-04   &   2.98\\
% \cline{2-8}
~&1/64& 1.01e-05   &   2.99   &   1.63e-05   &   3.00 \\
\cline{1-6}
\multirow{4}{*}{$(10,1)$}&
 1/8 & 4.91e-03   &   --   &   7.86e-03   &   -- \\
% \cline{2-8}
~& 1/16 &  6.70e-04   &   2.87  &   1.03e-03   &   2.93\\
% \cline{2-8}
~&1/32& 8.60e-05   &   2.96  &   1.30e-04   &   2.98\\
% \cline{2-8}
~&1/64& 1.08e-05   &   2.99  &   1.63e-05   &   3.00 \\
\cline{1-6}
\multirow{4}{*}{$(1,10^3)$}&
 1/8 & 4.58e-03   &   --  &   3.92e-03   &   -- \\
% \cline{2-8}
~& 1/16 & 6.24e-04   &   2.86   &   4.86e-04   &   3.01\\
% \cline{2-8}
~&1/32& 8.01e-05   &   2.96   &   6.07e-05   &   3.00\\
~&1/64& 1.01e-05   &   2.99   &   7.59e-06   &   3.00 \\
\cline{1-6}
\multirow{4}{*}{$(10^3,1)$}&
 1/8    & 4.92e-03   &   --   &    7.85e-03   &   -- \\
% \cline{2-8}
~& 1/16 &6.70e-04   &   2.88   &   1.03e-03   &   2.93\\
% \cline{2-8}
~&1/32  &8.60e-05   &   2.96   &   1.30e-04   &   2.98\\
% \cline{2-8}
~&1/64  &1.08e-05   &   2.99  &   1.63e-05   &   3.00  \\
\cline{1-6}
\multirow{4}{*}{$(1,10^5)$}&
 1/8    &4.59e-03   &   --     &   3.74e-03   &   -- \\
% \cline{2-8}
~& 1/16 &6.24e-04   &   2.88   &   4.59e-04   &   3.03\\
% \cline{2-8}
~&1/32 &8.01e-05   &   2.96    &   5.73e-05   &   3.00\\
% \cline{2-8}
~&1/64 &1.01e-05   &   2.99    &   7.16e-06   &   3.00 \\
\cline{1-6}
\multirow{4}{*}{$(10^5,1)$}&
 1/8 &5.59e-03    &   --     &   7.85e-03   &   -- \\
% \cline{2-8}
~& 1/16 &7.30e-04 &   2.94   &   1.03e-03   &   2.93\\
% \cline{2-8}
~&1/32&9.04e-05   &   3.01   &   1.30e-04   &   2.98\\
% \cline{2-8}
~&1/64&1.13e-05   &   3.00  &   1.63e-05   &   3.00 \\
\cline{1-6}
\end{tabular}}
\end{center}
\end{table}

\begin{table}[!htbp]
\caption{\small\textit{Errors $\|\bl e_u\|_{0,h}$ and $\| e_p\|_{0,\Omega_h}$ of Example 2 with boundary correction and respect to different parameters $\kappa$ and $k=1$.}}
\label{tab:eg2_sin_k1}
\begin{center}
{\small
\begin{tabular}{c|c | c| c| c|c}
\hline
\multirow{2}{*}{$(\kappa_1,\kappa_2)$}&\multirow{2}{*}{$h$}&\multicolumn{4}{c}{$k=1$}\\
\cline{3-6}
 &~ & $\|\bl e_u\|_{0,h}$  & order  & $\| e_p\|_{0,\Omega_h}$ & order \\
\cline{1-6}
\multirow{4}{*}{$(1,10)$}&
 1/8 & 1.13e-02   &    --    &   4.06e-01   &    --  \\
% \cline{2-6}
~& 1/16 &  6.00e-03   &   0.92   &   2.04e-01   &   0.99\\
% \cline{2-6}
~&1/32& 3.02e-03   &   0.99   &   1.02e-01   &   1.00\\
% \cline{2-6}
~&1/64&  1.51e-03   &   1.00   &   5.10e-02   &   1.00 \\
\cline{1-6}
\multirow{4}{*}{$(10,1)$}&
1/8   &   1.11e-02   &   --    &   3.61e-01   &   --    \\
~& 1/16    &   5.83e-03   &   0.92   &   1.77e-01   &   1.03   \\
~& 1/32    &   2.93e-03   &   0.99   &   8.76e-02   &   1.02   \\
~& 1/64    &   1.46e-03   &   1.00   &   4.36e-02   &   1.01   \\
\cline{1-6}
\multirow{4}{*}{$(1,10^3)$}&
1/8&     1.14e-02   &   --   &   4.04e-01   &  --   \\
~&1/16  &   6.02e-03   &   0.92   &   2.04e-01   &   0.98   \\
~&1/32   &   3.03e-03   &   0.99   &   1.02e-01   &   1.00   \\
~&1/64   &   1.51e-03   &   1.00   &   5.11e-02   &   1.00   \\
\cline{1-6}
\multirow{4}{*}{$(10^3,1)$}&
1/8   &   1.11e-02   &   --    &   3.41e-01   &  --    \\
~&1/16   &   5.83e-03   &   0.93   &   1.72e-01   &   0.99   \\
~&1/32   &   2.93e-03   &   0.99   &   8.60e-02   &   1.00   \\
~&1/64   &   1.47e-03   &   1.00   &   4.30e-02   &   1.00   \\
\cline{1-6}
\multirow{4}{*}{$(1,10^5)$}&
1/8   &   1.14e-02   &   --    &   4.05e-01   &  --    \\
~&1/16   &   6.02e-03   &   0.92   &   2.04e-01   &   0.99   \\
~&1/32  &   3.03e-03   &   0.99   &   1.02e-01   &   1.00   \\
~&1/64   &   1.52e-03   &   1.00   &   5.11e-02   &   1.00   \\
\cline{1-6}
\multirow{4}{*}{$(10^5,1)$}&
1/8   &   1.11e-02   &   --    &   3.41e-01   &  --    \\
~&1/16   &   5.83e-03   &   0.93   &   1.72e-01   &   0.99   \\
~&1/32   &   2.94e-03   &   0.99   &   8.60e-02   &   1.00   \\
~&1/64    &   1.47e-03   &   1.00   &   4.30e-02   &   1.00   \\
\cline{1-6}
\end{tabular}}
\end{center}
\end{table}

\begin{table}[!htbp]
\caption{\small\textit{Errors $\|\bl e_u\|_{0,h}$ and $\| e_p\|_{0,\Omega_h}$ of Example 2 with boundary correction and respect to different parameters $\kappa$ and $k=2$.}}
\label{tab:eg2_sin_k2}
\begin{center}
{\small
\begin{tabular}{c|c | c| c| c|c}
\hline
\multirow{2}{*}{$(\kappa_1,\kappa_2)$}&\multirow{2}{*}{$h$}&\multicolumn{4}{c}{$k=2$}\\
\cline{3-6}
 &~ & $\|\bl e_u\|_{0,h}$  & order  & $\| e_p\|_{0,\Omega_h}$ & order \\
\cline{1-6}
\multirow{4}{*}{$(1,10)$}&
 1/8   &   2.86e-03   &   --    &   9.69e-02   &   --   \\
~& 1/16   &   7.52e-04   &   1.93   &   2.53e-02   &   1.94   \\
~&1/32    &   1.90e-04   &   1.99   &   6.39e-03   &   1.98   \\
~&1/64    &   4.74e-05   &   2.00   &   1.60e-03   &   2.00   \\
\cline{1-6}
\multirow{4}{*}{$(10,1)$}&
 1/8  &   2.74e-03   &   --    &   7.86e-02   &   --    \\
~& 1/16   &   7.18e-04   &   1.93   &   2.04e-02   &   1.95   \\
~&1/32   &   1.81e-04   &   1.99   &   5.15e-03   &   1.99   \\
~&1/64   &   4.52e-05   &   2.00   &   1.29e-03   &   2.00   \\
\cline{1-6}
\multirow{4}{*}{$(1,10^3)$}&
1/8    &   3.02e-03   &   --    &   9.66e-02   &  --    \\
~& 1/16    &   7.91e-04   &   1.94   &   2.53e-02   &   1.93   \\
~& 1/32   &   1.98e-04   &   2.00   &   6.39e-03   &   1.98   \\
~& 1/64   &   4.87e-05   &   2.02   &   1.60e-03   &   2.00   \\

\cline{1-6}
\multirow{4}{*}{$(10^3,1)$}&
1/8   &   2.75e-03   &   --    &   7.73e-02   &   --    \\
~& 1/16   &   7.16e-04   &   1.94   &   2.01e-02   &   1.95   \\
~& 1/32   &   1.80e-04   &   1.99   &   5.07e-03   &   1.99   \\
~& 1/64   &   4.51e-05   &   2.00   &   1.27e-03   &   2.00   \\
\cline{1-6}
\multirow{4}{*}{$(1,10^5)$}&

1/8   & 1.69e-02   &   --    &   9.65e-02   &   --   \\
~&1/16   &   3.99e-03   &   2.09   &   2.55e-02   &   1.92   \\
~&1/32  &   1.03e-03   &   1.95   &   6.41e-03   &   1.99   \\
~&1/64   & 2.51e-04   &   2.04   &   1.62e-03   &   1.99   \\ 
\cline{1-6}
\multirow{4}{*}{$(10^5,1)$}&
1/8   &   1.46e-02   &   --    &   7.73e-02   &   --   \\
~&1/16   &   3.84e-03   &   1.93   &   2.02e-02   &   1.94   \\
~&1/32   &   9.80e-04   &   1.97   &   5.08e-03   &   1.99   \\
~&1/64    &   2.45e-04   &   2.00   &   1.28e-03   &   1.99   \\
\cline{1-6}
\end{tabular}}
\end{center}
\end{table}

\begin{table}[!htbp]
\caption{\small\textit{Errors $\|\bl e_u\|_{0,h}$ and $\| e_p\|_{0,\Omega_h}$ of Example 2 with boundary correction and respect to different parameters $\kappa$ and $k=3$.}}
\label{tab:eg2_sin_k3}
\begin{center}
{\small
\begin{tabular}{c|c | c| c| c|c}
\hline
\multirow{2}{*}{$(\kappa_1,\kappa_2)$}&\multirow{2}{*}{$h$}&\multicolumn{4}{c}{$k=3$}\\
\cline{3-6}
 &~ & $\|\bl e_u\|_{0,h}$  & order  & $\| e_p\|_{0,\Omega_h}$ & order \\
\cline{1-6}
\multirow{4}{*}{$(1,10)$}&
 1/8   &   5.46e-04   &   --    &   1.90e-02   &   --    \\
~& 1/16    &   7.21e-05   &   2.92   &   2.49e-03   &   2.94   \\
~&1/32    &   9.12e-06   &   2.98   &   3.16e-04   &   2.98   \\
~&1/64    &   1.14e-06   &   3.00   &   3.96e-05   &   2.99   \\
\cline{1-6}
\multirow{4}{*}{$(10,1)$}&
 1/8     &   5.14e-04   &   --   &   1.38e-02   &  --    \\
~& 1/16   &   6.78e-05   &   2.92   &   1.83e-03   &   2.92   \\
~&1/32  &   8.57e-06   &   2.98   &   2.32e-04   &   2.98   \\
~&1/64   &   1.07e-06   &   3.00   &   2.91e-05   &   2.99   \\
\cline{1-6}
\multirow{4}{*}{$(1,10^3)$}&
1/8  &   6.13e-04   &   --    &   1.92e-02   &   --    \\
~& 1/16    &   7.58e-05   &   3.01   &   2.49e-03   &   2.94   \\
~&1/32    &   9.31e-06   &   3.03   &   3.16e-04   &   2.98   \\
~&1/64    &   1.15e-06   &   3.01   &   3.97e-05   &   2.99   \\

\cline{1-6}
\multirow{4}{*}{$(10^3,1)$}&
1/8   &   5.17e-04   &   --    &   1.35e-02   &   --   \\
~& 1/16   &   6.75e-05   &   2.94   &   1.79e-03   &   2.92   \\
~&1/32    &   8.53e-06   &   2.99   &   2.27e-04   &   2.98   \\
~&1/64   &   1.07e-06   &   3.00   &   2.84e-05   &   2.99   \\
\cline{1-6}
\multirow{4}{*}{$(1,10^5)$}&

1/8    &   4.19e-03   &   --   &   1.92e-02   &  --    \\
~&1/16   &   4.81e-04   &   3.12   &   2.49e-03   &   2.94   \\
~&1/32   &   5.65e-05   &   3.09   &   3.16e-04   &   2.98   \\
~&1/64   &   6.90e-06   &   3.03   &   3.97e-05   &   2.99   \\  
\cline{1-6}
\multirow{4}{*}{$(10^5,1)$}&
1/8    &   3.21e-03   &   --    &   1.35e-02   &   --   \\
~&1/16   &   4.21e-04   &   2.93   &   1.79e-03   &   2.92   \\
~&1/32   &   5.34e-05   &   2.98   &   2.27e-04   &   2.98   \\
~&1/64   &   6.70e-06   &   2.99   &   2.84e-05   &   2.99   \\
\cline{1-6}
\end{tabular}}
\end{center}
\end{table}

\begin{table}[!htbp]
\caption{\small\textit{Errors $\|\bl e_u\|_{0,h}$ and $\| e_p\|_{0,\Omega_h}$ of Example 2 without boundary correction and respect to different parameters $\kappa$ and $k=3$.}}
\label{tab:eg2_sin_k31}
\begin{center}
{\small
\begin{tabular}{c|c | c| c| c|c}
\hline
\multirow{2}{*}{$(\kappa_1,\kappa_2)$}&\multirow{2}{*}{$h$}&\multicolumn{4}{c}{$k=3$}\\
\cline{3-6}
 &~ & $\|\bl e_u\|_{0,h}$  & order  & $\| e_p\|_{0,\Omega_h}$ & order \\
\cline{1-6}
\multirow{4}{*}{$(1,10)$}&
 1/8   &   9.44e-04   &   --   &   2.20e-02   &   -- \\
~& 1/16   &   6.16e-04   &   0.62   &   3.85e-03   &   2.52   \\
~& 1/32  &   4.41e-04   &   0.48   &   8.02e-04   &   2.26   \\
~& 1/64   &   3.13e-04   &   0.49   &   1.88e-04   &   2.09   \\
\cline{1-6}
\multirow{4}{*}{$(10,1)$}&
 1/8    &   9.21e-04   &   --    &   1.71e-02   &   --   \\
~& 1/16   &   6.15e-04   &   0.58   &   3.23e-03   &   2.40   \\
~& 1/32   &   4.41e-04   &   0.48   &   7.13e-04   &   2.18   \\
~& 1/64   &   3.13e-04   &   0.49   &   1.72e-04   &   2.05   \\
\cline{1-6}
\end{tabular}}
\end{center}
\end{table}

\section{Conclusions}\label{sec7:conclusion}

We have investigated the weak Galerkin mixed finite element method for elliptic interface equations with curved interfaces, where a boundary value corrections technique has been used to avoid numerical integration formula on curved elements. The non-homogeneous interface condition is weakly imposed in the variational formulation using the penalty method.
The parameters in the weak Galerkin formulation can be chosen arbitrarily without affecting the approximation results.
With this work, an optimal convergence rate in the energy norm is obtained.

\end{document}